\documentclass{article}

\usepackage{
  amsmath,
  amsthm,
  amssymb,
  mathtools,
  mathpazo,
  bm,
}
\usepackage{hyperref}
\usepackage{enumerate}
\usepackage{color}

\newtheorem{thm}{Theorem}
\newtheorem{prop}[thm]{Proposition}
\newtheorem{lem}[thm]{Lemma}
\newtheorem{cor}[thm]{Corollary}
\newtheorem{fact}[thm]{Fact}
\theoremstyle{definition}
\newtheorem*{rem*}{Remark}
\newtheorem*{note*}{Note}

\usepackage{tikz}

\title{Plane partitions with bounded size of parts and biorthogonal polynomials}
\author{ksh94}

\begin{document}
\maketitle

\begin{abstract}
  Nice formulae for plane partitions with bounded size of parts (or boxed plane partitions),
  which generalize the norm-trace generating function by Stanley and the trace generating function by Gansner,
  are exhibited.
  The derivation of the nice formulae is based on lattice path combinatorics of biorthogonal polynomials,
  especially of the little $q$-Laguerre polynomials and a generalization of the little $q$-Laguerre polynomials.
  A summation formula which generalizes the $q$-Chu--Vandermonde identity is also shown and utilized to prove
  the orthogonality of the generalized little $q$-Laguerre polynomials.
\end{abstract}

\section{Introduction}
\label{sec:introduction}

A plane partition $\pi$ of a nonnegative integer $N$ is a two-dimensional array
\begin{align}
  \pi = 
  \begin{pmatrix}
    \pi_{1,1} & \pi_{1,2} & \pi_{1,3} & \cdots \\
    \pi_{2,1} & \pi_{2,2} & \pi_{2,3} & \cdots \\
    \pi_{3,1} & \pi_{3,2} & \pi_{3,3} & \cdots \\
    \vdots    & \vdots    & \vdots    &
  \end{pmatrix}
\end{align}
of nonnegative integers $\pi_{i,j}$ such that $\sum_{i,j=1}^{\infty} \pi_{i,j} = N$ and
$\pi_{i,j} \ge \max\{ \pi_{i+1,j}, \pi_{i,j+1} \}$ for every $(i,j) \in \mathbb{Z}_{\ge 1}^{2}$.
(Throughout the paper we write $\mathbb{Z}_{\ge k}$ for the set of integers at least $k$.)
A plane partition $\pi$ distributes $N$ among its parts $\pi_{i,j}$ so that
each row and each column are non-increasing.
MacMahon studies plane partitions in depth and finds
the {\em norm generating function} for plane partitions (of rectangular shape) with {\em bounded size of parts}
\begin{align} \label{eq:NormGFBSP}
  \sum_{\pi \in \mathcal{P}(r,c,n)} q^{|\pi|}
  = \prod_{i=0}^{r-1} \prod_{j=0}^{c-1} \prod_{k=0}^{n-1} \frac{1 - q^{i+j+k+2}}{1 - q^{i+j+k+1}}
\end{align}
where $\mathcal{P}(r,c,n)$ denotes the set of plane partitions
of at most $r$ rows and at most $c$ columns with parts at most $n$, namely
$\pi \in \mathcal{P}(r,c,n)$ if and only if
$\pi_{r+i,j} = \pi_{i,c+j} = 0$ for every $(i,j) \in \mathbb{Z}_{\ge 1}$ and $\pi_{1,1} \le n$,
and $|\pi| = \sum_{i,j=1}^{\infty} \pi_{i,j}$ that is called the {\em norm} of $\pi$.
The limit $n \to \infty$ reduces \eqref{eq:NormGFBSP} to
the norm generating function for plane partitions with {\em unbounded size of parts}
\begin{align} \label{eq:NormGFUSP}
  \sum_{\pi \in \mathcal{P}(r,c)} q^{|\pi|}
  = \prod_{i=0}^{r-1} \prod_{j=0}^{c-1} (1 - q^{i+j+1})^{-1}
\end{align}
where $\mathcal{P}(r,c)$ denotes the set of plane partitions of at most $r$ rows and at most $c$ columns.
(No restriction is imposed to the size of parts.)
See MacMahon's book \cite[Section IX]{MacMahon(1916)} for details.

Generalizing the norm generating functions \eqref{eq:NormGFBSP} and \eqref{eq:NormGFUSP} is
an important subject in the study of plane partitions.
In particular a great progress is made by Stanley who considers the {\em trace} of plane partitions,
\begin{align}
  \mathsf{tr}(\pi) = \sum_{i=1}^{\infty} \pi_{i,i},
\end{align}
and finds the {\em norm-trace generating function}
\begin{align} \label{eq:NormTrGF}
  \sum_{\pi \in \mathcal{P}(r,c)} q^{|\pi|} a^{\mathsf{tr}(\pi)}
  = \prod_{i=0}^{r-1} \prod_{j=0}^{c-1} (1 - aq^{i+j+1})^{-1}
\end{align}
that recovers \eqref{eq:NormGFUSP} with $a = 1$ \cite{Stanley(1971),Stanley(1973)}.
Gansner later considers the {\em $\ell$-traces} of plane partitions,
\begin{align}
  \mathsf{tr}_{\ell}(\pi) = \sum_{j-i=\ell} \pi_{i,j}, \qquad \ell \in \mathbb{Z},
\end{align}
and extends \eqref{eq:NormTrGF} into the {\em trace generating function}
\begin{align} \label{eq:LTrGF}
  \sum_{\pi \in \mathcal{P}(r,c)} \prod_{-r < \ell < c} q_{\ell}^{\mathsf{tr}_{\ell}(\pi)}
  = \prod_{i=0}^{r-1} \prod_{j=0}^{c-1} \left( 1 - \prod_{\ell=-i}^{j} q_{\ell} \right)^{-1}
\end{align}
that recovers \eqref{eq:NormTrGF} with $q_{\ell} = q$ for every $\ell \in \mathbb{Z}$ except for $q_0 = aq$
\cite{Gansner(1981Burge),Gansner(1981HG)}.
(Gansner provides in \cite{Gansner(1981Burge),Gansner(1981HG)} more general results on
(reverse) plane partitions of arbitrary shape.)

Both the generating functions \eqref{eq:NormTrGF} and \eqref{eq:LTrGF} generalize
the norm generating function \eqref{eq:NormGFUSP} for plane partitions with unbounded size of parts.
We now have a simple question:
{\em Are there analogues of \eqref{eq:NormTrGF} and \eqref{eq:LTrGF} which generalize
the norm generating function \eqref{eq:NormGFBSP} for plane partitions with bounded size of parts}?
A naive answer is {\em no} because the simple replacement of $\mathcal{P}(r,c)$ in
\eqref{eq:NormTrGF} and \eqref{eq:LTrGF} with $\mathcal{P}(r,c,n)$ does not results in nice (product) formulae.
As is clarified in this paper (Theorems \ref{thm:NFPPBSO01} and \ref{thm:NFPPBSP02}), however,
the answer may be {\em yes} when we seek those from sums of the forms
\begin{align}
  \sum_{\pi \in \mathcal{P}(r,c,n)} q^{|\pi|} a^{\mathsf{tr}(\pi)} \omega_{n}(\pi)
  \quad \text{and} \quad
  \sum_{\pi \in \mathcal{P}(r,c,n)} \omega_{n}(\pi) \prod_{-r < \ell < c} q_{\ell}^{\mathsf{tr}_{\ell}(\pi)}
\end{align}
that include some weight functions $\omega_{n}$ for plane partitions such that
$\omega_{n}(\pi) \to 1$ as $n \to \infty$.

Orthogonal polynomials appear in various areas of mathematics, see, e.g., \cite{Szego(1975OP),Chihara(1978OP)}.
In combinatorics many combinatorial interpretations are given to
various families of (classical) orthogonal polynomials \cite{Foata(1984)}.
A combinatorial theory for general orthogonal polynomials is also given by Viennot who develops
a unified combinatorial approach to orthogonal polynomials by means of path diagrams
\cite{Viennot(1983OP),Viennot(1985)}.
Analogous results are also obtained for biorthogonal polynomials \cite{KimD(1992)}
(which are different from, in precise, a special case of biorthogonal polynomials examined in this paper) and for
Laurent biorthogonal polynomials \cite{Kamioka(2007),Kamioka(2008)}.
In this paper a combinatorial interpretation to general {\em biorthogonal polynomials} is developed in terms of
(weighted) lattice paths on a square lattice (Section \ref{sec:LPs}).

Non-intersecting paths and determinants are fundamental tools for analyzing plane partitions
\cite{Gessel-Viennot(PRE1989),Krattenthaler(1990),Johansson(2002)}.
In this paper the combinatorial interpretation of biorthogonal polynomials is applied to
deriving nice formulae for plane partitions with bounded size of parts where
exact evaluations of determinants are performed by means of biorthogonal polynomials.
Specifically the {\em little $q$-Laguerre polynomials} and their generalizations are examined to derive
nice formulae generalizing the trace-type generating functions \eqref{eq:NormTrGF} and \eqref{eq:LTrGF}.

This paper is organized as follows.
In Section \ref{sec:BOPs} basics of biorthogonal polynomials needed in this paper are described.
In Section \ref{sec:LPs} a combinatorial interpretation of (general) biorthogonal polynomials is developed
in terms of lattice paths on a square lattice.

In Section \ref{sec:LQLP} the combinatorial interpretation is applied to
the {\em little $q$-Laguerre polynomials} as a concrete example.
The results are utilized in Section \ref{sec:NFPPBSP01} to derive
a nice formula for plane partitions with bounded size of parts which generalizes
the norm-trace generating function \eqref{eq:NormTrGF} for those with unbounded size of parts
(Theorem \ref{thm:NFPPBSO01}).
The discussion in Section \ref{sec:LQLP} is generalized in Section \ref{sec:GenLQLP} for
the {\em generalized little $q$-Laguerre polynomials} newly introduced there, and
the results are used in Section \ref{sec:NFPPBSP02} to derive a nice formula
which generalizes the trace generating function \eqref{eq:LTrGF}
(Theorem \ref{thm:NFPPBSP02}).


The orthogonality of the generalized little $q$-Laguerre polynomials
(Theorem \ref{thm:GenLQLPOrthty}) is proven by means of a generalization of
the $q$-Chu--Vandermonde identity \cite{Ismail(2005CQOP),Koekoek-Leskey-Swarttow(2010)} for
basic hypergeometric series (Lemma \ref{lem:GenQCV}).
The proof of the lemma is given in Appendix \ref{sec:GenQCV}.

\section{Biorthogonal polynomials}
\label{sec:BOPs}

Let $\mathbb{K}$ be a field.
Let $\mathcal{F}: \mathbb{K}[x^{\pm 1},y^{\pm 1}] \to \mathbb{K}$ be a linear functional defined on
the space of Laurent polynomials in $x$ and $y$ over $\mathbb{K}$.
Due to the linearity, $\mathcal{F}$ is uniquely determined by the \emph{moments}
\begin{align} \label{eq:Moments}
  f_{i,j} = \mathcal{F}[x^i y^j], \qquad (i,j) \in \mathbb{Z}^{2}.
\end{align}
Let us define determinants of moments
\begin{align} 
  \Delta^{(r,c)}_{n}
  = & \det_{0 \le i,j < n} (f_{r+i,c+j}) \notag \\
  = &
  \begin{vmatrix}
    f_{r,c}     & \cdots & f_{r,c+j}     & \cdots & f_{r,c+n-1}     \\
    \vdots      &        & \vdots        &        & \vdots          \\
    f_{r+i,c}   & \cdots & f_{r+i,c+j}   & \cdots & f_{r+i,c+n-1}   \\
    \vdots      &        & \vdots        &        & \vdots          \\
    f_{r+n-1,c} & \cdots & f_{r+n-1,c+j} & \cdots & f_{r+n-1,c+n-1} \\
  \end{vmatrix}
\end{align}
for $(r,c) \in \mathbb{Z}^{2}$ and $n \in \mathbb{Z}_{\ge 0}$ where $\Delta^{(r,c)}_{0} = 1$.
We assume throughout the paper that the determinant $\Delta^{(r,c)}_{n}$ does not vanish.

We define a (monic) \emph{biorthogonal polynomial} $P^{(r,c)}_{n}(x) \in \mathbb{K}[x]$,
$(r,c) \in \mathbb{Z}^{2}$, $n \in \mathbb{Z}_{\ge 0}$, (with respect to $\mathcal{F}$) as a polynomial such that
the leading term of $P^{(r,c)}_{n}(x)$ is $x^{n}$ and the {\em orthogonality}
\begin{align} \label{eq:BOPOrthty}
  \mathcal{F}[x^{r} y^{c+j} P^{(r,c)}_{n}(x)] = h^{(r,c)}_{n} \delta_{j,n}, \qquad 0 \le j \le n,
\end{align}
holds with some normalization constant $h^{(r,c)}_{n} \in \mathbb{K} \setminus \{ 0 \}$ where
$\delta_{j,n}$ denotes the Kronecker delta.
The biorthogonal polynomial $P^{(r,c)}_{n}(x)$ is uniquely determined from $\mathcal{F}$.
Indeed $P^{(r,c)}_{n}(x)$ should have the determinant expression
\begin{align} \label{eq:BOPDet}
  P^{(r,c)}_{n}(x) = {}
  \begin{vmatrix}
    f_{r,c}   & \cdots & f_{r,c+j}   & \cdots & f_{r,c+n-1}   & 1      \\
    \vdots    &        & \vdots      &        & \vdots        & \vdots \\
    f_{r+i,c} & \cdots & f_{r+i,c+j} & \cdots & f_{r+i,c+n-1} & x^{i}  \\
    \vdots    &        & \vdots      &        & \vdots        & \vdots \\
    f_{r+n,c} & \cdots & f_{r+n,c+j} & \cdots & f_{r+n,c+n-1} & x^{n}  \\
  \end{vmatrix}
  {} \times (\Delta^{(r,c)}_{n})^{-1}
\end{align}
from the monicity and the orthogonality \eqref{eq:BOPOrthty}.
(Write down \eqref{eq:BOPOrthty} in a linear system for the coefficients of $P^{(r,c)}_{n}(x)$ and apply Cramer's rule.)
We have from \eqref{eq:BOPOrthty} and \eqref{eq:BOPDet} that
\begin{align} \label{eq:BOPNormConstDet}
  h^{(r,c)}_{n} = \frac{\Delta^{(r,c)}_{n+1}}{\Delta^{(r,c)}_{n}}.
\end{align}

The reason why we call $P^{(r,c)}_{n}(x)$ ``biorthogonal polynomials'' is the following.
Let us consider a monic polynomial in $y$ 
\begin{align}
  Q^{(r,c)}_{n}(y) = {}
  \begin{vmatrix}
    f_{r,c}     & \cdots & f_{r,c+j}     & \cdots & f_{r,c+n}     \\
    \vdots      &        & \vdots        &        & \vdots        \\
    f_{r+i,c}   & \cdots & f_{r+i,c+j}   & \cdots & f_{r+i,c+n}   \\
    \vdots      &        & \vdots        &        & \vdots        \\
    f_{r+n-1,c} & \cdots & f_{r+n-1,c+j} & \cdots & f_{r+n-1,c+n} \\
    1           & \cdots & y^{j}         & \cdots & y^{n}         \\
  \end{vmatrix}
  {} \times (\Delta^{(r,c)}_{n})^{-1}.
\end{align}
We then have the {\em biorthogonality} between $P^{(r,c)}_{n}(x)$ and $Q^{(r,c)}_{n}(x)$ that
\begin{align}
  \mathcal{F}[x^{r} y^{c} P^{(r,c)}_{m}(x) Q^{(r,c)}_{n}(y)] = h^{(r,c)}_{n} \delta_{m,n}, \qquad
  m,n \in \mathbb{Z}_{\ge 0}.
\end{align}
We thus call $P^{(r,c)}_{n}(x)$ biorthogonal polynomials in view of
the existsnece of biorthogonal partners $Q^{(r,c)}_{n}(y)$.

\begin{rem*}
  The biorthogonal polynomials considered here naturally involve (ordinary) orthogonal polynomials.
  Orthogonal polynomials, say $P_n(x)$, $n \in \mathbb{Z}_{\ge 0}$, are
  polynomials with $\deg P_n(x) = n$ satisfying the self-orthogonality
  \begin{align}
    \mathcal{F}[P_m(x) P_n(x)] = h_n \delta_{m,n}, \qquad m,n \in \mathbb{Z}_{\ge 0},
  \end{align}
  with some linear functional $\mathcal{F}: \mathbb{K}[x] \to \mathbb{K}$ and some constants $h_n \neq 0$, see, e.g.,
  \cite{Szego(1975OP),Chihara(1978OP)}.
  If we regard $x$ and $y$ as the same indeterminate $x=y$
  biorthogonal polynomials then reduce to orthogonal polynomials.
\end{rem*}

The following proposition plays a key role in
our combinatorial interpretation of biorthogonal polynomials in Section \ref{sec:LPs}.

\begin{prop}[cf.~\cite{Maeda-Miki-Tsujimoto(2013)}] \label{prop:ARs}
  The biorthogonal polynomials satisfy the adjacent relations
  \begin{subequations} \label{eq:ARs}
    \begin{align}
      x P^{(r+1,c)}_{n}(x) &= P^{(r,c)}_{n+1}(x) + a^{(r,c)}_{n} P^{(r,c)}_{n}(x),
      \label{eq:AR01} \\
      P^{(r,c)}_{n}(x) &= P^{(r,c+1)}_{n}(x) + b^{(r,c)}_{n} P^{(r,c+1)}_{n-1}(x)
      \label{eq:AR02}
    \end{align}
  \end{subequations}
  for $(r,c) \in \mathbb{Z}^{2}$ and $n \in \mathbb{Z}_{\ge 0}$ with $P^{(r,c+1)}_{-1}(x) \equiv 0$ where
  \begin{subequations}
    \begin{align}
      a^{(r,c)}_{n}
      {} = \frac{h^{(r+1,c)}_{n}}{h^{(r,c)}_{n}}
      {} = \frac{\Delta^{(r+1,c)}_{n+1} \Delta^{(r,c)}_{n}}{\Delta^{(r+1,c)}_{n} \Delta^{(r,c)}_{n+1}},
      \label{eq:ARCfA} \\
      b^{(r,c)}_{n}
      {} = \frac{h^{(r,c)}_{n}}{h^{(r,c+1)}_{n-1}}
      {} = \frac{\Delta^{(r,c)}_{n+1} \Delta^{(r,c+1)}_{n-1}}{\Delta^{(r,c)}_{n} \Delta^{(r,c+1)}_{n}}.
      \label{eq:ARCfB}
    \end{align}
  \end{subequations}
\end{prop}

\begin{proof}
  We expand $x P^{(r+1,c)}_{n}(x)$ into a linear combination of $P^{(r,c)}_{k}(x)$.
  Since $x P^{(r+1,c)}_{n}(x)$ is a monic polynomial in $x$ of degree $n+1$ we have
  \begin{align} \label{eq:ARsExpand}
    x P^{(r+1,c)}_{n}(x) = P^{(r,c)}_{n+1}(x) + \sum_{k=0}^{n} \alpha_k P^{(r,c)}_{k}(x)
  \end{align}
  with some constants $\alpha_k$ which can be determined as follows.
  Multiply $x^{r} y^{c}$ to the both sides of \eqref{eq:ARsExpand} and apply $\mathcal{F}$.
  We then obtain $\alpha_0 = 0$ from the orthogonality \eqref{eq:BOPOrthty} and hence
  \begin{align}
    x P^{(r+1,c)}_{n}(x) = P^{(r,c)}_{n+1}(x) + \sum_{k=1}^{n} \alpha_k P^{(r,c)}_{k}(x).
  \end{align}
  Next multiply $x^{r} y^{c+1}$ and apply $\mathcal{F}$ similarly.
  We then obtain $\alpha_1 = 0$ from \eqref{eq:BOPOrthty} and
  \begin{align}
    x P^{(r+1,c)}_{n}(x) = P^{(r,c)}_{n+1}(x) + \sum_{k=2}^{n} \alpha_k P^{(r,c)}_{k}(x),
  \end{align}
  and so on.
  We finally find that $\alpha_0 = \cdots = \alpha_{n-1} = 0$ and
  \begin{align}
    x P^{(r+1,c)}_{n}(x) = P^{(r,c)}_{n+1}(x) + \alpha_n P^{(r,c)}_{n}(x).
  \end{align}
  Here multiply $x^{r+n} y^{c}$ and apply $\mathcal{F}$ to find that
  $\alpha_n = h^{(r+1,c)}_{n} / h^{(r,c)}_{n} = a^{(r,c)}_{n}$ from \eqref{eq:BOPOrthty}.
  We thus obtain \eqref{eq:AR01}.
  The relation \eqref{eq:AR02} can be obtained in almost the same way as \eqref{eq:AR01} by using
  the orthogonality \eqref{eq:BOPOrthty}.
\end{proof}

\section{Lattice path combinatorics}
\label{sec:LPs}

We show in Section \ref{sec:LPs} a combinatorial interpretation of
(general) biorthogonal polynomials in terms of lattice paths on a square lattice.
The combinatorial interpretation partly owes the basic idea to
the combinatorial theory of general orthogonal polynomials developed by
Viennot \cite{Viennot(1983OP),Viennot(1985)}.

Let us view a two-dimensional integral lattice $\mathbb{Z}_{\ge 0}^{2}$ (in the first quadrant) as a square lattice.
We depict the square lattice in matrix-like coordinates where
the south and east neighbors of the point $(i,j) \in \mathbb{Z}_{\ge 0}^{2}$ are $(i+1,j)$ and $(i,j+1)$ respectively.
The square lattice $\mathbb{Z}_{\ge 0}^{2}$ makes a graph of vertical and horizontal edges
connecting every two neighboring lattice points.
Let $\alpha_{i,j} \in \mathbb{K}$ for $(i-1,j) \in \mathbb{Z}_{\ge 0}^{2}$.
We then label the vertical edge connecting $(i,j)$ and $(i-1,j)$ by $\alpha_{i,j}$ and
every horizontal edge by $1$, as shown in Figure \ref{fig:SqLattice}.
\begin{figure}
  \centering
  \includegraphics{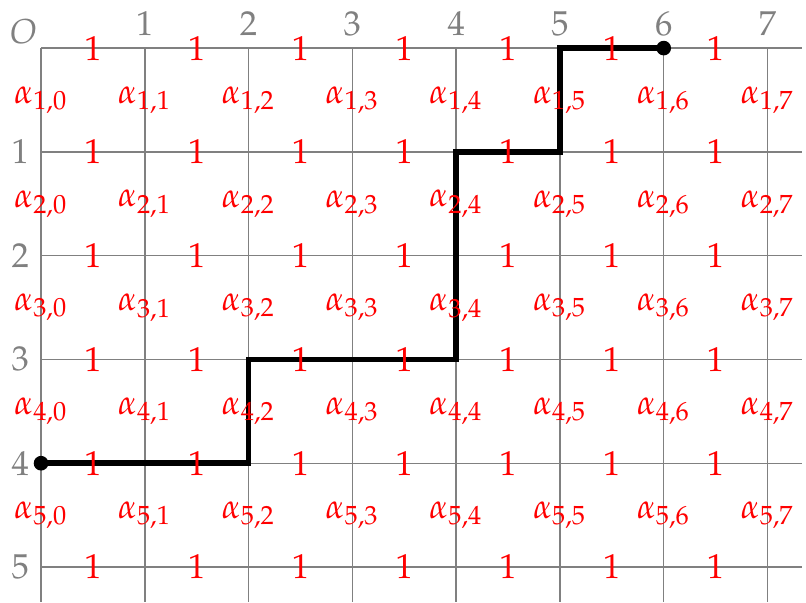}
  \caption{The square lattice $\mathbb{Z}_{\ge 0}^{2}$ with edges labelled.
    The thick line shows a lattice path on the square lattice going from $(4,0)$ to $(0,6)$.}
  \label{fig:SqLattice}
\end{figure}

Lattice paths considered in this paper are those on the square lattice $\mathbb{Z}_{\ge 0}^{2}$
which travel from some lattice point to another with north and south (unit) steps.
For example, a lattice path going from $(4,0)$ to $(0,6)$ is shown in Figure \ref{fig:SqLattice}.
Conventionally the two endpoints of a lattice path may coincide with each other so that
the lattice path is an {\em empty} path of no steps.
The {\em weight} of a lattice path $P$, $w(P)$, is defined to be
the product of the labels of all the edges passed by $P$.
For example, the lattice path in Figure \ref{fig:SqLattice} has the weight
$w(P) = \alpha_{4,2} \alpha_{3,4} \alpha_{2,4} \alpha_{1,5}$.
The weight of any empty lattice path is assume to be $1$.

The following theorem provides a foundation of our combinatorial interpretation of biorthogonal polynomials.

\begin{thm} \label{thm:MomentsLPaths}
  Assume that
  \begin{subequations} \label{eq:VEdgeLabels}
    \begin{align}
      \alpha_{i,j}
      & = a^{(i-j-1,0)}_{j} && \text{if $i > j$;}
      \label{eq:VEdgeLabels+} \\
      & = b^{(0,j-i)}_{i}   && \text{if $i \le j$}
      \label{eq:VEdgeLabels-}
    \end{align}
  \end{subequations}
  for each $(i-1,j) \in \mathbb{Z}_{\ge 0}^{2}$ where
  the right-hand sides are coefficients of the adjacent relations \eqref{eq:ARs} among biorthogonal polynomials.
  Let $(r,c) \in \mathbb{Z}_{\ge 0}^{2}$.
  The moments \eqref{eq:Moments} of biorthogonal polynomials then satisfy
  \begin{align} \label{eq:MomentsLPaths}
    \frac{f_{r,c}}{f_{0,c}} = \sum_{P} w(P)
  \end{align}
  where the sum ranges over
  all the lattice paths $P$ on the square lattice $\mathbb{Z}_{\ge 0}^{2}$ going from $(r,0)$ to $(0,c)$.
\end{thm}

\begin{proof}
  Let us assume \eqref{eq:VEdgeLabels} for labels on vertical edges.
  The formula \eqref{eq:MomentsLPaths} is then induced from
  \begin{align} \label{eq:BOPsLPaths}
    x^{r} P^{(r,0)}_{n}(x) = \sum_{k=0}^{\infty} P^{(0,c)}_{k}(x) \sum_{P^{(k)}} w(P^{(k)})
  \end{align}
  where $P^{(k)}$ in the second sum is over all the lattice paths going from $(r+n,n)$ to $(k,c+k)$.
  (The first sum with respect to $k = 0,1,2,\dots$ is actually a finite sum because,
  if $k > r+n$ or $k < n-c$, there are no lattice paths $P^{(k)}$ going from $(r+n,n)$ to $(k,c+k)$ and hence
  $\sum_{P^{(k)}} w(P^{(k)}) = 0$.)
  Indeed, when $n = 0$, we multiply $y^{c}$ to the both sides of \eqref{eq:BOPsLPaths} and
  apply $\mathcal{F}$ to obtain
  \begin{align}
    f_{r,c} = h^{(0,c)}_{0} \sum_{P^{(0)}} w(P^{(0)})
  \end{align}
  from the orthogonality \eqref{eq:BOPOrthty} where $P^{(0)}$ goes from $(r,0)$ to $(0,c)$.
  The last equation is equivalent to \eqref{eq:MomentsLPaths} since
  $h^{(0,c)}_{0} = f_{0,c}$ from \eqref{eq:BOPNormConstDet}.

  We now prove \eqref{eq:BOPsLPaths}.
  We first show that
  \begin{align} \label{eq:BOPsLPaths+}
    x^{r} P^{(r,0)}_{n}(x) = \sum_{j=0}^{\infty} P^{(0,0)}_{j}(x) \sum_{P^{(j)}_{+}} w(P^{(j)}_{+})
  \end{align}
  where $P^{(j)}_{+}$ runs over all the lattice paths going from $(r+n,n)$ to $(j,j)$ on the main diagonal.
  The formula \eqref{eq:BOPsLPaths+} can be proven by induction with respect to $r = 0,1,2,\dots$ as follows.
  If $r=0$ the formula \eqref{eq:BOPsLPaths+} reads
  \begin{align} \label{eq:BOPsLPaths+0}
    P^{(0,0)}_{n}(x) = P^{(0,0)}_{n}(x) w(P^{(n)}_{+})
  \end{align}
  where $P^{(n)}_{+}$ is the unique (empty) lattice path going from and to $(n,n)$.
  The equality \eqref{eq:BOPsLPaths+0} is surely true since $w(P^{(n)}_{+}) = 1$.
  Assume that $r \ge 1$.
  From the adjacent relation \eqref{eq:AR01} we then have
  \begin{align}
    x^{r} P^{(r,0)}_{n}(x) = x^{r-1} P^{(r-1,0)}_{n+1}(x) + \alpha_{r+n,n} x^{r-1} P^{(r-1,0)}_{n}(x)
  \end{align}
  where $\alpha_{r+n,n} = a^{(r-1,0)}_{n}$ from \eqref{eq:VEdgeLabels+}.
  The assumption of induction yields that
  \begin{align} \label{eq:BOPsLPaths+01}
    x^{r} P^{(r,0)}_{n}(x)
    = \sum_{j=0}^{\infty} P^{(0,0)}_{j}(x)
    \left( \sum_{P'^{(j)}_{+}} w(P'^{(j)}_{+}) + \alpha_{r+n,n} \sum_{P''^{(j)}_{+}} w(P''^{(j)}_{+}) \right)
  \end{align}
  where $P'^{(j)}_{+}$ and $P''^{(j)}_{+}$ run over
  all the lattice paths going from $(r+n,n+1)$ to $(j,j)$ and
  those from $(r+n-1,n)$ to $(j,j)$ respectively.
  The lattice paths going from $(r+n,n)$ to $(j,j)$ are classified into two classes:
  those starting with an east step, labelled by $1$, followed by
  a lattice path going from $(r+n,n+1)$ to $(j,j)$;
  those starting with a north step, labelled by $\alpha_{r+n,n}$, followed by
  a lattice path going from $(r+n-1,n)$ to $(j,j)$, see Figure \ref{fig:Classify}.
  \begin{figure}
    \centering
    \includegraphics{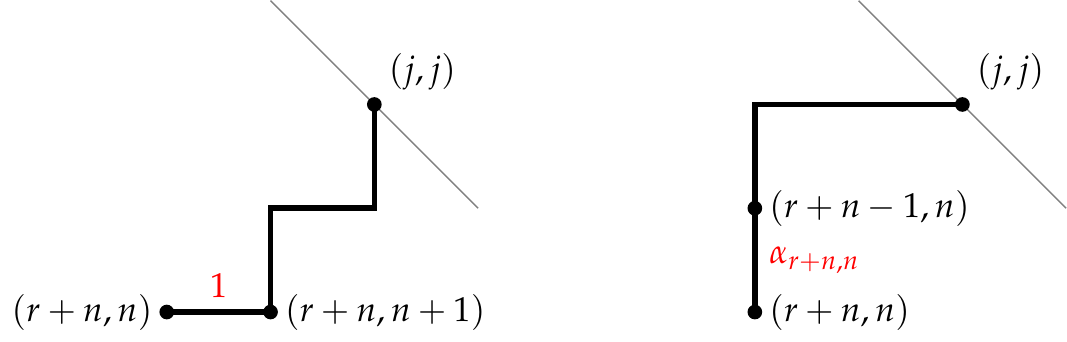}
    \caption{Classification of lattice paths going from $(r+n,n)$ to $(j,j)$ into
      those starting by an east step labelled by $1$ (left) and
      those starting by a north step labelled by $\alpha_{r+n,n}$ (right).}
    \label{fig:Classify}
  \end{figure}
  We can thereby unify the two sums for $P'^{(j)}_{+}$ and $P''^{(j)}_{+}$ in \eqref{eq:BOPsLPaths+01} into
  that for $P^{(j)}_{+}$ going from $(r+n,n)$ to $(j,j)$:
  \begin{align}
    \sum_{P'^{(j)}_{+}} w(P'^{(j)}_{+}) + \alpha_{r+n,n} \sum_{P''^{(j)}_{+}} w(P''^{(j)}_{+})
    = \sum_{P^{(j)}_{+}} w(P^{(j)}_{+}).
  \end{align}
  We thus obtain \eqref{eq:BOPsLPaths+}.
  We can show in a similar way that
  \begin{align} \label{eq:BOPsLPaths-}
    P^{(0,0)}_{j}(x) = \sum_{k=0}^{\infty} P^{(0,c)}_{k}(x) \sum_{P^{(j,k)}_{-}} w(P^{(j,k)}_{-})
  \end{align}
  by using \eqref{eq:AR02} and \eqref{eq:VEdgeLabels-} where
  $P^{(j,k)}_{-}$ runs over all the lattice paths going from $(j,j)$ to $(k,c+k)$.
  Substituting \eqref{eq:BOPsLPaths-} for \eqref{eq:BOPsLPaths+} we get
  \begin{align} \label{eq:BOPsLPaths00}
    x^{r} P^{(r,0)}_{n}(x) = \sum_{k=0}^{\infty} P^{(0,c)}_{k}(x)
    \sum_{j=0}^{\infty} \sum_{(P^{(j)}_{+},P^{(j,k)}_{-})} w(P^{(j)}_{+}) w(P^{(j,k)}_{-})
  \end{align}
  where $(P^{(j)}_{+},P^{(j,k)}_{-})$ ranges over all the pairs of
  a lattice path $P^{(j)}_{+}$ going from $(r+n,n)$ to $(j,j)$ and
  a lattice path $P^{(j,k)}_{-}$ going from $(j,j)$ to $(k,c+k)$.
  Note that we can concatenate $P^{(j)}_{+}$ and $P^{(j,k)}_{-}$ at $(j,j)$ to get
  a lattice path, say $P^{(j,k)}$, going from $(r+n,n)$ to $(k,c+k)$ via $(j,j)$ on the main diagonal.
  \begin{figure}
    \centering
    \includegraphics{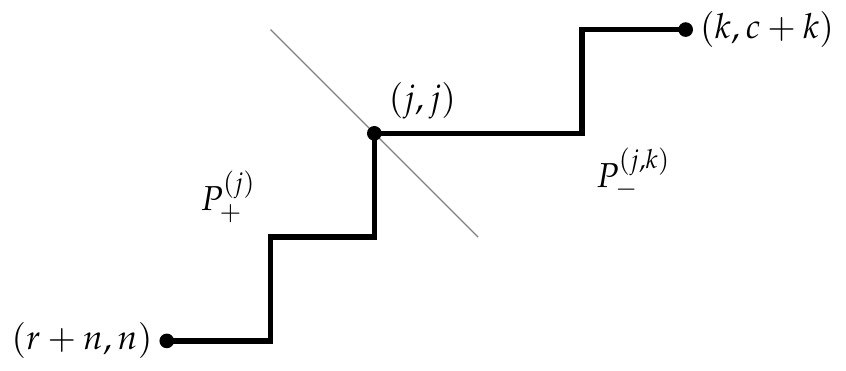}
    \caption{Concatenation of $P^{(j)}_{+}$ and $P^{(j,k)}_{-}$ at $(j,j)$ where
      $P^{(j)}_{+}$ goes from $(r+n,n)$ to $(j,j)$ while $P^{(j,k)}_{-}$ from $(j,j)$ to $(k,c+k)$.}
    \label{fig:Concatenation}
  \end{figure}
  Therefore, since $w(P^{(j)}_{+}) w(P^{(j,k)}_{-}) = w(P^{(j,k)})$,
  \begin{align} \label{eq:BOPsLPathWeights00}
    \sum_{j=0}^{\infty} \sum_{(P^{(j)}_{+},P^{(j,k)}_{-})} w(P^{(j)}_{+}) w(P^{(j,k)}_{-})
    = \sum_{j=0}^{\infty} \sum_{P^{(j,k)}} w(P^{(j,k)})
    = \sum_{P^{(k)}} w(P^{(k)})
  \end{align}
  where $P^{(k)}$ in the last sum runs over all the lattice paths going from $(r+n,n)$ to $(k,c+k)$.
  We thus obtain \eqref{eq:BOPsLPaths} from \eqref{eq:BOPsLPaths00} and \eqref{eq:BOPsLPathWeights00}.
  That completes the proof.
\end{proof}

Theorem \ref{thm:MomentsLPaths} provides
a combinatorial interpretation of moments of biorthogonal polynomials in terms of lattice paths on a square lattice.
The combinatorial interpretation of moments leads to
the following combinatorial interpretation of determinants of moments.
For $(r,c,n) \in \mathbb{Z}_{\ge 0}^{3}$ we define $\mathcal{LP}(r,c,n)$ to be the set of
$n$-tuples $(P_0,\dots,P_{n-1})$ of lattice paths on the square lattice $\mathbb{Z}_{\ge 0}^{2}$ such that
\begin{enumerate}[(i)]
\item
  $P_k$ goes from $(r+k,0)$ to $(0,c+k)$;
\item
  $P_0,\dots,P_{n-1}$ are {\em non-intersecting}, namely $P_j \cap P_k = \emptyset$ if $j \neq k$.
\end{enumerate}
Figure \ref{fig:NILPaths} shows an example of such an $n$-tuple $(P_0,\dots,P_{n-1}) \in \mathcal{LP}(r,c,n)$.
\begin{figure}
  \centering
  \includegraphics{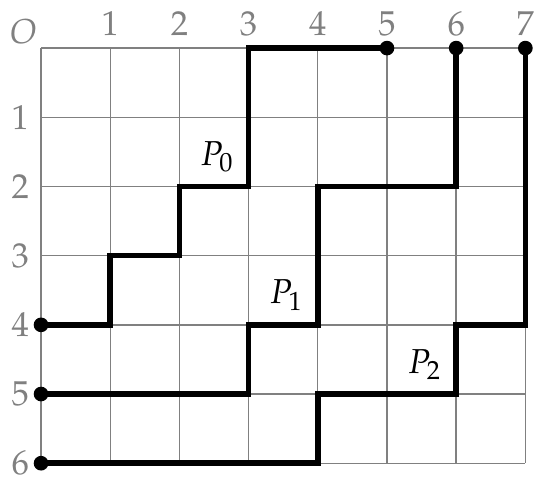}
  \caption{An $n$-tuple $(P_0,\dots,P_{n-1}) \in \mathcal{LP}(r,c,n)$ of
    non-intersecting lattice paths on the square lattice $\mathbb{Z}_{\ge 0}^{2}$ where $(r,c,n) = (4,5,3)$.}
  \label{fig:NILPaths}
\end{figure}

\begin{cor} \label{cor:DetNILPs}
  Let $(r,c,n) \in \mathbb{Z}_{\ge 0}^{3}$.
  Then
  \begin{align} \label{eq:DetNILPs}
    \frac{\Delta^{(r,c)}_{n}}{\prod_{k=0}^{n-1} f_{0,c+k}}
    = \sum_{(P_0,\dots,P_{n-1}) \in \mathcal{LP}(r,c,n)} \prod_{k=0}^{n-1} w(P_k).
  \end{align}
\end{cor}

\begin{proof}
  Theorem \ref{thm:MomentsLPaths} implies that
  \begin{align}
    \frac{\Delta^{(r,c)}_{n}}{\prod_{k=0}^{n-1} f_{0,c+k}}
    = \det_{0 \le i,j < n} \left( \sum_{P_{i,j}} w(P_{i,j}) \right)
  \end{align}
  where $P_{i,j}$ runs over all the lattice paths going from $(r+i,0)$ to $(0,c+j)$.
  We can directly equate the last determinant with the sum in the right-hand side of \eqref{eq:DetNILPs} by means of
  Gessel--Viennot--Lindstr\"om's method \cite{Gessel-Viennot(1985),Lindstroem(1973)}, see also
  \cite[Chapter 31]{Aigner-Guenter(2014TheBook)}.
\end{proof}

\section{Little $q$-Laguerre polynomials}
\label{sec:LQLP}

We examine in Section \ref{sec:LQLP} the {\em little $q$-Laguerre polynomials} as a concrete example of
the combinatorial interpretation of (general) biorthogonal polynomials discussed in Section \ref{sec:LPs}.
The results are applied in Section \ref{sec:NFPPBSP01} to deriving
a nice formula for plane partitions with bounded size of parts which generalizes
the norm-trace generating function \eqref{eq:NormTrGF} by Stanley \cite{Stanley(1971),Stanley(1973)}.

In what follows we adopt the following conventional notations for $q$-analysis:
$q$-Pochhammer symbols
\begin{subequations}
  \begin{align}
    (a;q)_{n} = \prod_{k=1}^{\infty} \frac{1-aq^{n-k}}{1-aq^{-k}}
    {} &= \prod_{k=0}^{n-1} (1-aq^{k})     && \text{if $n > 0$,} \\
    {} &= 1                                && \text{if $n = 0$,} \\
    {} &= \prod_{k=n}^{-1} (1-aq^{k})^{-1} && \text{if $n < 0$,}
  \end{align}
\end{subequations}
with abbreviation
\begin{align}
  (a_1,\dots,a_m;q)_{n} &= \prod_{j=1}^{m} (a_j;q)_{n},
\end{align}
and basic hypergeometric series
\begin{align}
  {}_{2}\phi{}_{1}\left( \genfrac{}{}{0pt}{}{a,b}{c}; q, x \right)
  {} = \sum_{j=0}^{\infty} \frac{(a,b;q)_{j}}{(c,q;q)_{j}} x^j.
\end{align}

The (monic) {\em little $q$-Laguerre polynomial} of degree $n \in \mathbb{Z}_{\ge 0}$ is given by
\begin{align} \label{eq:LQLP}
  L_{n}(x;a;q) = (-1)^{n} q^{\frac{n(n-1)}{2}} (aq;q)_{n} \times {}
  {}_{2}\phi{}_{1}\left( \genfrac{}{}{0pt}{}{q^{-n},0}{aq}; q, qx \right).
\end{align}
The little $q$-Laguerre polynomial is a classical orthogonal polynomials which resides in
the Askey scheme, see \cite[\S 14.20]{Koekoek-Leskey-Swarttow(2010)} and the references therein.
In this paper we think of the parameters $a$ and $q$ as independent indeterminates so that
$L_{n}(x;a;q) \in \mathbb{K}[x]$ with $\mathbb{K} = \mathbb{C}(a,q)$.
(The reader can instead think of $a$ and $q$ as complex constants such that
$0 < |q| < 1$ and $0 < |aq| < 1$ as in \cite{Koekoek-Leskey-Swarttow(2010)}.)

Let us fix the linear functional $\mathcal{F}: \mathbb{K}[x^{\pm 1},y^{\pm 1}] \to \mathbb{K}$ by the moments
\begin{align} \label{eq:LQLPMoments}
  f_{i,j} = \mathcal{F}[x^i y^j] = (aq^{j+1};q)_{i}, \qquad (i,j) \in \mathbb{Z}^{2}.
\end{align}
For $(r,c) \in \mathbb{Z}^{2}$ and $n \in \mathbb{Z}_{\ge 0}$ let us write
\begin{align} \label{eq:LQLPBOPs}
  L^{(r,c)}_{n}(x;a;q) = L_{n}(x;aq^{r+c};q).
\end{align}
The orthogonality \eqref{eq:BOPOrthty} and the adjacent relations \eqref{eq:ARs} for
the little $q$-Laguerre polynomials are given as follows.

\begin{prop} \label{prop:LQLPOrthty}
  The little $q$-Laguerre polynomial $L^{(r,c)}_{n}(x;a;q)$ satisfies the orthogonality \eqref{eq:BOPOrthty} with
  $P^{(r,c)}_{n}(x) = L^{(r,c)}_{n}(x;a;q)$ and
  \begin{align} \label{eq:LQLPNormConst}
    h^{(r,c)}_{n} = f_{r,c+n} \times a^{n} q^{n(r+c+n)} (q;q)_{n}
  \end{align}
  with respect to the linear functional $\mathcal{F}$ having the moments \eqref{eq:LQLPMoments} where
  $f_{r,c+n} = \mathcal{F}[x^{r} y^{c+n}]$.
\end{prop}

\begin{proof}
  From \eqref{eq:LQLP}--\eqref{eq:LQLPBOPs} we have
  \begin{align} \label{eq:LQLPOrthty00}
    \frac{\mathcal{F}[x^{r} y^{c+j} L^{(r,c)}_{n}(x;a;q)]}{(-1)^{n} q^{\frac{n(n-1)}{2}} (aq^{r+c+1};q)_{n}}
    & = (aq^{c+j+1};q)_{r} \sum_{i=0}^{n} \frac{(q^{-n},aq^{r+c+j+1};q)_{i}}{(q,aq^{r+c+1};q)_{i}} q^{i} \notag \\
    & = f_{r,c+j} \times {}_{2}\phi{}_{1}\left( \genfrac{}{}{0pt}{}{q^{-n},aq^{r+c+j+1}}{aq^{r+c+1}}; q, q \right).
  \end{align}
  We here apply the $q$-Chu--Vandermonde identity
  \begin{align} \label{eq:QCV}
    {}_{2}\phi{}_{1}\left( \genfrac{}{}{0pt}{}{q^{-n},a}{c}; q, q \right) = \frac{a^{n} (a^{-1} c;q)_{n}}{(c;q)_{n}}
  \end{align}
  \cite[Eq.~(1.11.5)]{Koekoek-Leskey-Swarttow(2010)}, \cite[Theorem 12.2.4]{Ismail(2005CQOP)} to
  the last hypergeometric series to get
  \begin{align} \label{eq:LQLPQCV}
    {}_{2}\phi{}_{1}\left( \genfrac{}{}{0pt}{}{q^{-n},aq^{r+c+j+1}}{aq^{r+c+1}}; q, q \right)
    & = \frac{a^{n} q^{n(r+c+j+1)} (q^{-j};q)_{n}}{(aq^{r+c+1};q)_{n}} \notag \\
    & = \frac{a^{n} q^{n(r+c+n+1)} (q^{-n};q)_{n}}{(aq^{r+c+1};q)_{n}} \delta_{j,n}
  \end{align}
  for $0 \le j \le n$.
  Substituting \eqref{eq:LQLPQCV} for \eqref{eq:LQLPOrthty00}
  we obtain the orthogonality with \eqref{eq:LQLPNormConst}.
\end{proof}

\begin{rem*}
  The (self-)orthogonality of the little $q$-Laguerre polynomials is usually described by
  \begin{align} \label{eq:LQLPOrthtyOnQGrid}
    \sum_{j=0}^{\infty} L_{m}(q^j;a;q) L_{n}(q^j;a;q) \frac{(aq)^{j}}{(q;q)_{j}}
    = \frac{a^{n} q^{n^2} (q;q)_{n}}{(aq^{n+1};q)_{\infty}} \delta_{m,n}, \qquad
    m,n \in \mathbb{Z}_{\ge 0},
  \end{align}
  \cite[\S 14.20]{Koekoek-Leskey-Swarttow(2010)}.
  The orthogonality \eqref{eq:LQLPOrthtyOnQGrid} can be equivalently written as
  \begin{align} \label{eq:LQLPOrthtyAsOPs}
    \mathcal{F}'[L_{m}(x;a;q) L_{n}(x;a;q)]
    = \frac{a^{n} q^{n^2} (q;q)_{n}}{(aq^{n+1};q)_{\infty}} \delta_{m,n}
  \end{align}
  with the linear functional $\mathcal{F}': \mathbb{K}[x] \to \mathbb{K}$ determined by
  the moments $\mathcal{F}'[x^i] = (aq;q)_{i}$.
  It is easy to see that \eqref{eq:LQLPOrthtyAsOPs} is equivalent to
  the orthogonality stated in Proposition \ref{prop:LQLPOrthty}.
\end{rem*}

\begin{cor} \label{cor:LQLPARs}
  The little $q$-Laguerre polynomials $L^{(r,c)}_{n}(y;a;q)$ satisfy the adjacent relations \eqref{eq:ARs} with
  $Q^{(r,c)}_{n}(y) = L^{(r,c)}_{n}(y;a;q)$ and
  \begin{subequations} \label{eq:LQLPARCfs}
    \begin{align}
      a^{(r,c)}_{n} &= q^{n} (1-aq^{r+c+n+1}), \\
      b^{(r,c)}_{n} &= a q^{r+c+n} (1-q^{n}).
    \end{align}
  \end{subequations}
\end{cor}

\begin{proof}
  That is immediate from Propositions \ref{prop:ARs} and \ref{prop:LQLPOrthty}.
\end{proof}

The lattice path combinatorics for biorthogonal polynomials, discussed in Section \ref{sec:LPs}, is applied to
the little $q$-Laguerre polynomials as follows.
In view of Theorem \ref{thm:MomentsLPaths} and Corollary \ref{cor:LQLPARs}
we label the vertical edge between $(i,j)$ and $(i-1,j)$ in the square lattice $\mathbb{Z}_{\ge 0}^{2}$ by
\begin{subequations} \label{eq:LQLPLabelling}
  \begin{align}
    \alpha_{i,j}
    & = q^{j} (1-aq^{i})  && \text{if $i > j$;} \\
    & = a q^{j} (1-q^{i}) && \text{if $i \le j$,}
  \end{align}
\end{subequations}
and every horizontal edge by $1$.
In the rest of Section \ref{sec:LQLP} and in Section \ref{sec:NFPPBSP01}
we the weights of lattice paths are evaluated with respect to this specific labelling.

Let $(r,c) \in \mathbb{Z}_{\ge 0}^{2}$.
Let $P$ be a lattice path on the square lattice $\mathbb{Z}_{\ge 0}^{2}$ going from $(r,0)$ to $(0,c)$.
Viewing the finite region bordered by $P$ as a Young diagram we can naturally identify $P$ with
an (integer) partition of at most $r$ parts whose parts are at most $c$.
We write $\lambda(P)$ for the partition.
For example, the lattice path $P$ in Figure \ref{fig:SqLattice} is identified with
the partition $\lambda(P) = (5,4,4,2)$.

Let $\lambda$ be a partition.
The {\em norm} $|\lambda|$ is equal to the number of boxes contained in the Young diagram of $\lambda$.
The {\em Durfee square} is a maximal square that can be contained in a Young diagram.
We define $\mathsf{D}(\lambda)$ to be the size of the Durfee square of the Young diagram of $\lambda$.
Obviously $\mathsf{D}(\lambda)$ is equal to the number of boxes on the main diagonal of $\lambda$.
We note that $\mathsf{D}(\lambda(P)) = d$ if and only if
the lattice path $P$ passes through $(d,d)$ on the main diagonal.

\begin{lem} \label{lem:LQLPPathWeight}
  Let $(r,c) \in \mathbb{Z}_{\ge 0}^{2}$ and
  let $P$ be a lattice path on the square lattice $\mathbb{Z}_{\ge 0}^{2}$ going from $(r,0)$ to $(0,c)$.
  The weight $w(P)$ with respect to the labelling given by \eqref{eq:LQLPLabelling} then admits that
  \begin{subequations} \label{eq:LQLPPathWeight}
    \begin{align}
      \frac{w(P)}{(aq;q)_{r}} & = q^{|\lambda(P)|} a^{\mathsf{D}(\lambda(P))} \omega(P;a;q) \qquad \text{where} \\
      \omega(P;a;q) & = \frac{(q;q)_{\mathsf{D}(\lambda(P))}}{(aq;q)_{\mathsf{D}(\lambda(P))}}.
      \label{eq:LQLPOmega}
    \end{align}
  \end{subequations}
\end{lem}

\begin{proof}
  We ``factor'' the labelling given by \eqref{eq:LQLPLabelling} into two distinct labellings;
  one puts on the vertical edge between $(i,j)$ and $(i-1,j)$
  \begin{subequations} \label{eq:LQLPLabel01}
    \begin{align}
      \alpha'_{i,j}
      & = q^{j}   && \text{if $i > j$;} \\
      & = a q^{j} && \text{if $i \le j$,}
    \end{align}
  \end{subequations}
  and the other
  \begin{subequations} \label{eq:LQLPLabel02}
    \begin{align}
      \alpha''_{i,j}
      & = 1-aq^{i} && \text{if $i > j$;} \\
      & = 1-q^{i}  && \text{if $i \le j$.}
    \end{align}
  \end{subequations}
  We write $w'(P)$ and $w''(P)$ for the weights of a lattice path $P$ with respect to
  the labellings given by \eqref{eq:LQLPLabel01} and \eqref{eq:LQLPLabel02} respectively.
  Obviously $w(P) = w'(P) w''(P)$.
  Let $P$ be a lattice path mentioned in the lemma.
  It is easily seen from \eqref{eq:LQLPLabel01} and \eqref{eq:LQLPLabel02} that
  \begin{align}
    w'(P) = q^{|\lambda(P)|} a^{\mathsf{D}(\lambda(P))}
  \end{align}
  and
  \begin{align}
    w''(P)
    = \left\{ \prod_{j=1}^{d} (1-q^j) \right\} \left\{ \prod_{j=d+1}^{r} (1-aq^j) \right\}
    = \frac{(q;q)_{d} (aq;q)_{r}}{(aq;q)_{d}}
  \end{align}
  respectively where $d = \mathsf{D}(\lambda(P))$.
  We therefore have \eqref{eq:LQLPPathWeight}.
\end{proof}

Theorem \ref{thm:MomentsLPaths} and Lemma \ref{lem:LQLPPathWeight} imply the following.

\begin{thm} \label{thm:LQLPMomentsLPaths}
  Let $\mathcal{F}$ be the linear functional (for the little $q$-Laguerre polynomials) determined by
  the moments \eqref{eq:LQLPMoments}.
  Let $(r,c) \in \mathbb{Z}_{\ge 0}^{2}$.
  Then
  \begin{align} \label{eq:LQLPMomentsLPaths}
    \frac{f_{r,c}}{(aq;q)_{r}} = \sum_{P} q^{|\lambda(P)|} a^{\mathsf{D}(\lambda(P))} \omega(P;a;q)
  \end{align}
  where the sum ranges over
  all the lattice paths $P$ on the square lattice $\mathbb{Z}_{\ge 0}^{2}$ going from $(r,0)$ to $(0,c)$, and
  $\omega(P;a;q)$ is given by \eqref{eq:LQLPOmega}.
\end{thm}

\begin{proof}
  Delete $w(P)$ from \eqref{eq:MomentsLPaths} and \eqref{eq:LQLPPathWeight} to get the result where
  $f_{0,c} = (aq^{c+1};q)_{0} = 1$ from \eqref{eq:LQLPMoments}.
\end{proof}

The left-hand side of \eqref{eq:LQLPMomentsLPaths} is equal to
\begin{align}
  \frac{f_{r,c}}{(aq;q)_{r}}
  = \frac{(aq^{c+1};q)_{r}}{(aq;q)_{r}}
  = \frac{(aq;q)_{r+c}}{(aq;q)_{r} (aq;q)_{c}}
\end{align}
that generalizes the $q$-binomial coefficient
\begin{align}
  \genfrac{[}{]}{0pt}{}{r+c}{r}_{q} = \frac{(q;q)_{r+c}}{(q;q)_{r} (q;q)_{c}}.
\end{align}
The formula \eqref{eq:LQLPMomentsLPaths} thereby gives a generalization of the well-known formula
\begin{align}
  \genfrac{[}{]}{0pt}{}{r+c}{r}_{q} = \sum_{P} q^{|\lambda(P)|}
\end{align}
for a combinatorial interpretation of the $q$-binomial coefficients \cite[\S 1.6]{Aigner(2007CE)}.

Corollary \ref{cor:DetNILPs} and Lemma \ref{lem:LQLPPathWeight} imply the following.

\begin{thm} \label{thm:LQLPDetNILP}
  Let $\mathcal{F}$ be the linear functional (for the little $q$-Laguerre polynomials) determined by
  the moments \eqref{eq:LQLPMoments}.
  Let $(r,c,n) \in \mathbb{Z}_{\ge 0}^{3}$.
  Then
  \begin{align} \label{eq:LQLPDetNILPs}
    \frac{\Delta^{(r,c)}_{n}}{\prod_{k=0}^{n-1} (aq;q)_{r+k}}
    = \sum_{(P_0,\dots,P_{n-1}) \in \mathcal{LP}(r,c,n)}
    q^{\sum_{k=0}^{n-1} |\lambda(P_k)|} a^{\sum_{k=0}^{n-1} \mathsf{D}(\lambda(P_k))}
    \prod_{k=0}^{n-1} \omega(P_k;a;q)
  \end{align}
  where $\omega(P_k;a;q)$ is given by \eqref{eq:LQLPOmega}.
\end{thm}

\begin{proof}
  Combine \eqref{eq:DetNILPs} and \eqref{eq:LQLPPathWeight} to get the result where
  $f_{0,c} = (aq^{c+1};q)_{0} = 1$ from \eqref{eq:LQLPMoments}.
\end{proof}

We note that the right-hand side of \eqref{eq:LQLPDetNILPs} is equal to the determinant
\begin{align}
  \frac{\Delta^{(r,c)}_{n}}{\prod_{k=0}^{n-1} (aq;q)_{r+k}}
  = \det_{0 \le i,j < n}\left( \frac{(aq;q)_{r+c+i+j}}{(aq;q)_{r+i} (aq;q)_{c+j}} \right)
\end{align}
that generalizes the $q$-binomial determinant
\begin{align}
  \det_{0 \le i,j < n}\left( \genfrac{[}{]}{0pt}{}{r+c+i+j}{r+i}_{q} \right).
\end{align}

\section{Nice formula for plane partitions with bounded size of parts, I}
\label{sec:NFPPBSP01}

It is customary to depict a plane partition $\pi = (\pi_{i,j})_{i,j = 1,2,3,\dots}$ in
a {\em three-dimensional (3D) Young diagram} in which
$\pi_{i,j}$ (unit) cubes are stacked over the position $(i,j)$.
For example, the plane partition
\begin{align} \label{eq:PPEx}
  \begin{pmatrix}
    3 & 3 & 3 & 2 & 2 \\
    3 & 3 & 3 & 1 & 1 \\
    3 & 3 & 2 & 1 & 0 \\
    3 & 2 & 0 & 0 & 0 \\
  \end{pmatrix}
\end{align}
is depicted as the 3D Young diagram shown in Figure \ref{fig:3DYD}.
\begin{figure}
  \centering
  \includegraphics{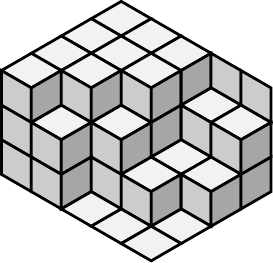}
  \caption{The 3D Young diagram corresponding to the plane partition \eqref{eq:PPEx}.}
  \label{fig:3DYD}
\end{figure}
The {\em norm} $|\pi|$ is then equal to the number of cubes stacked in the 3D Young diagram of $\pi$.

As is mentioned in Section \ref{sec:introduction} Stanley finds
the {\em norm-trace generating function} for plane partitions with {\em unbounded size of parts}
\begin{align} \label{eq:NormTrGFRep}
  \sum_{\pi \in \mathcal{P}(r,c)} q^{|\pi|} a^{\mathsf{tr}(\pi)}
  = \prod_{i=0}^{r-1} \prod_{j=0}^{c-1} (1 - aq^{i+j+1})^{-1}
\end{align}
where $\mathcal{P}(r,c)$ denote the set of plane partitions of at most $r$ rows and at most $c$ columns, and
$\mathsf{tr}(\pi)$ the {\em trace} of $\pi$ \cite{Stanley(1971),Stanley(1973)}.
Based on the results on the little $q$-Laguerre polynomials in Section \ref{sec:LQLP}
we find in Section \ref{sec:NFPPBSP01} a nice formula for plane partitions with {\em bounded size of parts}
which is analogous to \eqref{eq:NormTrGFRep} and generalizes the norm generating function \eqref{eq:NormGFBSP} for
those with bounded size of parts by MacMahon \cite{MacMahon(1916)}.

Let $(r,c,n) \in \mathbb{Z}_{\ge 0}^{3}$.
The $\mathcal{P}(r,c,n)$ denotes the set of plane partitions of at most $r$ rows and at most $c$ columns
whose parts are at most $n$.
For example, the plane partition \eqref{eq:PPEx} belongs to $\mathcal{P}(r,c,n)$ if and only if
$r \ge 4$, $c \ge 5$ and $n \ge 3$.
In other words $\pi \in \mathcal{P}(r,c,n)$ if and only if
the 3D Young diagram of $\pi$ is confined in an $r \times c \times n$ rectangular box.

In view of 3D Young diagrams it is so natural to characterize plane partitions by means of
(integer) partitions as follows.
Let $\pi$ be a plane partition.
For each $k \in \mathbb{Z}_{\ge 1}$ we define a partition $\lambda_k(\pi)$ by
the cross-section at level $k$ of the 3D Young diagram of $\pi$.
For example, the plane partition \eqref{eq:PPEx}, or the 3D Young diagram in Figure \ref{fig:3DYD}, gives rise to
the partitions
\begin{align}
  \lambda_1(\pi) = (5,5,4,2), \qquad
  \lambda_2(\pi) = (5,3,3,2), \qquad
  \lambda_3(\pi) = (3,3,2,1)
\end{align}
and $\lambda_k(\pi) = \emptyset = (0,0,0,\dots)$ for $k \ge 4$, see Figure \ref{fig:3DYDCrossSec}.
\begin{figure}
  \centering
  \includegraphics{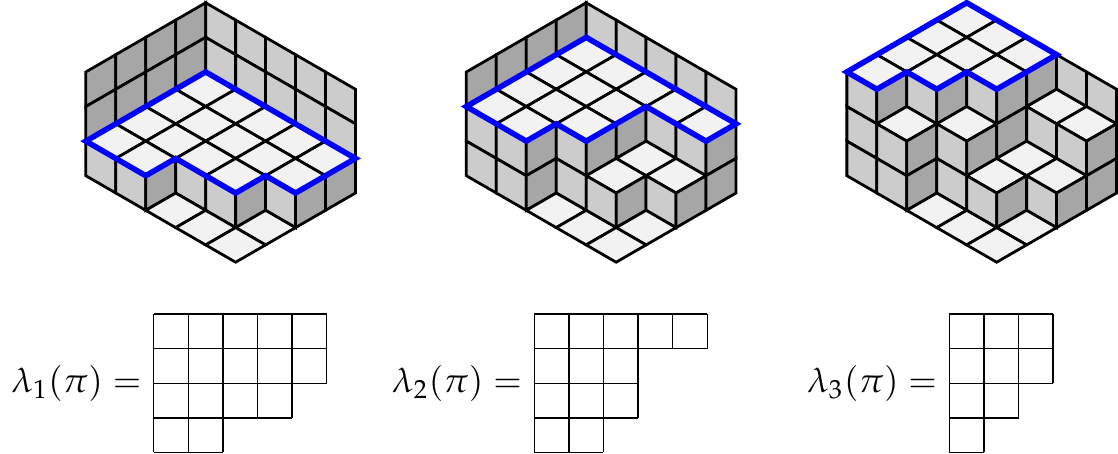}
  \caption{The cross-sections $\lambda_k(\pi)$ at level $k = 1,2,3$ of
    the 3D Young diagram $\pi$ in Figure \ref{fig:3DYD}.}
  \label{fig:3DYDCrossSec}
\end{figure}
Another characterization of $\lambda_k(\pi)$ given as:
the $i$-th part of $\lambda_k(\pi)$ is equal to the number of parts in the $i$-th row of $\pi$ which are at least $k$.
The map $\pi \mapsto (\lambda_1(\pi),\lambda_2(\pi),\lambda_3(\pi),\dots)$ is clearly a bijection between
plane partitions and sequences $(\lambda_1,\lambda_2,\lambda_3,\dots)$ of partitions such that
$\lambda_1 \supset \lambda_2 \supset \lambda_3 \supset \cdots \supset \lambda_{M} = \emptyset$ for
some $M \in \mathbb{Z}_{\ge 1}$ where
$\lambda \supset \mu$ means that $\lambda$ totally contains or coincides with $\mu$ as a Young diagram.
Obviously
\begin{align}
  |\pi| & = \sum_{k=1}^{\pi_{1,1}} |\lambda_k(\pi)|,
  \label{eq:PP-IP_Norm} \\
  \mathsf{tr}(\pi) & = \sum_{k=1}^{\pi_{1,1}} \mathsf{D}(\lambda_k(\pi))
  \label{eq:PP-IP_Trace}
\end{align}
where $\pi_{1,1}$ denotes the $(1,1)$-part of a plane partition $\pi$, and
$\mathsf{D}(\lambda)$ the size of the Durfee square of a partition $\lambda$.

We now recall a well-known bijection between $\mathcal{P}(r,c,n)$ and $\mathcal{LP}(r,c,n)$,
see Section \ref{sec:LPs} for the definition of $\mathcal{LP}(r,c,n)$.
Let $\lambda \in \mathcal{P}(r,c,n)$.
For each integer $k$, $0 \le k < n$,
draw on the square lattice $\mathbb{Z}_{\ge 0}^{2}$ a lattice path $P_k$ going from $(r+k,0)$ to $(0,c+k)$ such that
\begin{align} \label{eq:Bijection}
  \lambda(P_k) = (\underbrace{c,\dots,c}_{\text{$k$ times}},\lambda_{n-k,1}(\pi),\dots,\lambda_{n-k,r}(\pi)) + (k^{r+k})
\end{align}
where $\lambda_{n-k,i}(\pi)$ denotes the $i$-th part of the partition $\lambda_{n-k}(\pi)$, and
$(k^m) = (k,\dots,k)$ of $m$ parts equal to $k$.
Graphically speaking, the lattice path $P_k$ is obtained in the following procedure:
\begin{enumerate}[(i)]
\item
  Draw on $\mathbb{Z}_{\ge 0}^{2}$ a lattice path $P'_k$ going from $(r,0)$ to $(0,c)$ such that
  $\lambda(P'_k) = \lambda_{n-k}(\pi)$.
\item
  Translate $P'_k$ by $(k,k)$ (so that $P'_k$ goes from $(r+k,k)$ to $(k,c+k)$).
  Let us write $P''_k$ for the obtained lattice path.
\item
  Add $k$ consecutive east and north steps to the initial and terminal points of $P''_k$ respectively
  (so that $P''_k$ goes from $(r+k,0)$ to $(0,c+k)$).
  The obtained lattice path is $P_k$.
\end{enumerate}
For example, the procedure works as shown in Figure \ref{fig:Bijection} for
the plane partition \eqref{eq:PPEx} or for the 3D Young diagram in Figure \ref{fig:3DYD}.
\begin{figure}
  \centering
  \includegraphics{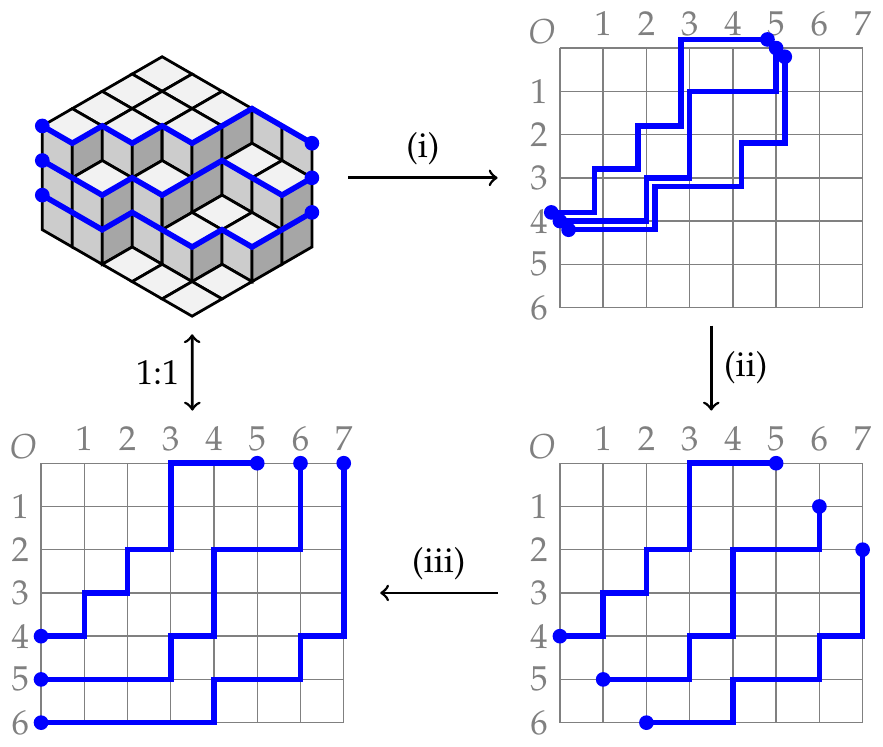}
  \caption{The bijection between $\mathcal{P}(r,c,n)$ and $\mathcal{LP}(r,c,n)$ with $(r,c,n) = (4,5,3)$.}
  \label{fig:Bijection}
\end{figure}
The relation that $\lambda_1(\pi) \supset \cdots \supset \lambda_n(\pi)$ guarantees that
the obtained lattice paths $P_0,\dots,P_{n-1}$ are non-intersecting and hence
$(P_0,\dots,P_{n-1}) \in \mathcal{LP}(r,c,n)$.
This procedure thus gives a map from $\mathcal{P}(r,c,n)$ to $\mathcal{LP}(r,c,n)$.

The above map from $\mathcal{P}(r,c,n)$ to $\mathcal{LP}(r,c,n)$ is invertible.
In fact, for any $(P_0,\dots,P_{n-1}) \in \mathcal{LP}(r,c,n)$,
the non-intersecting condition forces $P_k$ to start and end by $k$ consecutive east and north steps respectively.
So the inverse of (iii) in the procedure, of removing the $k$ consecutive east and north steps,
can be safely performed for any $(P_0,\dots,P_{n-1}) \in \mathcal{LP}(r,c,n)$.
There are no difficulties on the inverses of (ii) and (i), and
the non-intersecting condition for $(P_0,\dots,P_{n-1})$ guarantees that
we surely obtain a 3D Young diagram after applying the inverses of (iii), (ii) and (i).
The procedure therefore gives a bijection between $\mathcal{P}(r,c,n)$ and $\mathcal{LP}(r,c,n)$.

Suppose that a plane partition $\pi \in \mathcal{P}(r,c,n)$ and
an $n$-tuple $(P_0,\dots,P_{n-1}) \in \mathcal{LP}(r,c,n)$ of non-intersecting lattice paths
correspond to each other by the bijection.
It is immediate from \eqref{eq:Bijection} that
\begin{subequations} \label{eq:BjNormTr}
  \begin{align}
    |\lambda(P_k)| & = |\lambda_{n-k}(\pi)| + k(r+c+k), \\
    \mathsf{D}(\lambda(P_k)) & = \mathsf{D}(\lambda_{n-k}(\lambda)) + k.
  \end{align}
  Therefore
  \begin{align}
    \sum_{k=0}^{n-1} |\lambda(P_k)| & = |\pi| + \frac{n(n-1)(3r+3c+2n-1)}{6}, \\
    \sum_{k=0}^{n-1} |\mathsf{D}(\lambda(P_k))| & = \mathsf{tr}(\pi) + \frac{n(n-1)}{2}.
  \end{align}
\end{subequations}

The bijection between $\mathcal{P}(r,c,n)$ and $\mathcal{LP}(r,c,n)$, with \eqref{eq:BjNormTr}, allows us to
translate Theorem \ref{thm:LQLPDetNILP} in the language of plane partitions.

\begin{thm} \label{thm:NFPPBSO01}
  Let $(r,c,n) \in \mathbb{Z}_{\ge 0}^{3}$. Then
  \begin{subequations} \label{eq:NFPPBSO01}
    \begin{align}
      \label{eq:NFPPBSO01Sum}
      \sum_{\pi \in \mathcal{P}(r,c,n)} q^{|\pi|} a^{\mathsf{tr}(\pi)} \omega_{n}(\pi;a;q)
      & = \prod_{i=0}^{r-1} \prod_{j=0}^{c-1} \prod_{k=0}^{n-1} \frac{1-aq^{i+j+k+2}}{1-aq^{i+j+k+1}} \qquad
      \text{where} \\
      \label{eq:NFPPBSO01_Wgt}
      \omega_{n}(\pi;a;q)
      & = \prod_{k=1}^{\pi_{1,1}}
      \frac{(q^{n-k+1};q)_{\mathsf{D}(\lambda_k(\pi))}}{(aq^{n-k+1};q)_{\mathsf{D}(\lambda_k(\pi))}}
    \end{align}
  \end{subequations}
  where $\pi_{1,1}$ denotes the $(1,1)$-part of a plane partition $\pi$.
\end{thm}

\begin{proof}
  The bijection between $\mathcal{P}(r,c,n)$ and $\mathcal{LP}(r,c,n)$,
  with the help of \eqref{eq:BjNormTr}, allows us to
  equivalently rewrite the formula \eqref{eq:LQLPDetNILPs} in Theorem \ref{thm:LQLPDetNILP} as follows:
  \begin{align} \label{eq:NFPPBSO0100}
    \sum_{\pi \in \mathcal{P}(r,c,n)} q^{|\pi|} a^{\mathsf{tr}(\pi)} \omega_{n}(\pi;a;q)
    = \frac{\Delta^{(r,c)}_{n}}{q^{\frac{n(n-1)(3r+3c+2n-1)}{6}} a^{\frac{n(n-1)}{2}}.
      \prod_{k=0}^{n-1} (aq^{k+1};q)_{r} (q;q)_{k}}    
  \end{align}
  Note that $\mathsf{D}(\lambda_k(\pi)) = 0$ for $k > \pi_{1,1}$.
  The proof thus amounts to the evaluation of
  the determinant $\Delta^{(r,c)}_{n}$ of moments \eqref{eq:LQLPMoments} of
  the little $q$-Laguerre polynomials.
  From \eqref{eq:BOPNormConstDet} we have
  \begin{align} \label{eq:DetByNormConst}
    \Delta^{(r,c)}_{n} = \prod_{k=0}^{n-1} h^{(r,c)}_{k}
  \end{align}
  for general biorthogonal polynomials.
  We therefore find from
  the normalization constant \eqref{eq:LQLPNormConst} for the little $q$-Laguerre polynomials that
  \begin{align}
    \Delta^{(r,c)}_{n}
    = q^{\frac{n(n-1)(3r+3c+2n-1)}{6}} a^{\frac{n(n-1)}{2}} \prod_{k=0}^{n-1} (aq^{c+k+1};q)_{r} (q;q)_{k}.
  \end{align}
  Substituting the last equation for \eqref{eq:NFPPBSO0100} we obtain
  \begin{align}
    \sum_{\pi \in \mathcal{P}(r,c,n)} q^{|\pi|} a^{\mathsf{tr}(\pi)} \omega_{n}(\pi;a;q)
    = \prod_{k=0}^{n-1} \frac{(aq^{c+k+1};q)_{r}}{(aq^{k+1};q)_{r}}.
  \end{align}
  The last product is equal to the right-hand side of \eqref{eq:NFPPBSO01Sum}.
\end{proof}

The nice formula \eqref{eq:NFPPBSO01} for plane partitions with bounded size of parts generalizes
the norm-trace generating function \eqref{eq:NormTrGFRep} for those with unbounded size of parts.
Indeed, $\mathcal{P}(r,c,n) \to \mathcal{P}(r,c)$, $\omega_{n}(\pi;a;q) \to 1$ and
\begin{align}
  \prod_{i=0}^{r-1} \prod_{j=0}^{c-1} \prod_{k=0}^{n-1} \frac{1-aq^{i+j+k+2}}{1-aq^{i+j+k+1}}
  = \prod_{i=0}^{r-1} \prod_{j=0}^{c-1} \frac{1-aq^{n+i+j+1}}{1-aq^{i+j+1}}
  \to \prod_{i=0}^{r-1} \prod_{j=0}^{c-1} (1-aq^{i+j+1})^{-1}
\end{align}
as $n \to \infty$ since $\lim_{n \to \infty} q^{n} = 0$ as a formal power series in $q$
(or as a complex number with $|q| < 1$).
The nice formula \eqref{eq:NFPPBSO01} also recovers
the norm generating function \eqref{eq:NormGFBSP} for plane partitions with bounded size of parts with $a = 1$ since
$\omega_{n}(\pi;1;q) \equiv 1$ from \eqref{eq:NFPPBSO01_Wgt}.

\section{Generalized little $q$-Laguerre polynomials}
\label{sec:GenLQLP}

We show in Section \ref{sec:GenLQLP} another concrete example of
the combinatorial interpretation of biorthogonal polynomials discussed in Section \ref{sec:LPs}.
We introduce a generalization of the little $q$-Laguerre polynomials and examine
the lattice path combinatorics of those.
The results in Section \ref{sec:GenLQLP} are utilized in Section \ref{sec:NFPPBSP02} to derive
a nice formula for plane partitions with bounded size of parts which generalizes
the trace generating function \eqref{eq:LTrGF} for those with unbounded size of parts.

In what follows we use the following notation:
For any bilateral sequence $\bm{x} = (\dots,x_{-2},x_{-1},x_0,x_1,x_2,\dots)$ and $n \in \mathbb{Z}$,
\begin{subequations}
  \begin{align}
    \bm{x}^{\overline{n}} = \prod_{k=0}^{\infty} \frac{x_{n-k}}{x_{-k}}
    & = \prod_{k=1}^{n} x_{k}        && \text{if $n > 0$;} \\
    & = 1                            && \text{if $n = 0$;} \\
    & = \prod_{k=n+1}^{0} x_{k}^{-1} && \text{if $n < 0$.}
  \end{align}
\end{subequations}

Let $a$ be an indeterminate and let
\begin{subequations}
  \begin{align}
    \bm{p} &= (\dots,p_{-2},p_{-1},p_{0},p_{1},p_{2},\dots), \\
    \bm{q} &= (\dots,q_{-2},q_{-1},q_{0},q_{1},q_{2},\dots)
  \end{align}
\end{subequations}
be bilateral sequences of indeterminates $p_i$ and $q_j$.
We write
\begin{subequations}
  \begin{align}
    \bm{p}_{m} &= (\dots,p_{m-2},p_{m-1},p_{m},p_{m+1},p_{m+2},\dots), \\
    \bm{q}_{m} &= (\dots,q_{m-2},q_{m-1},q_{m},q_{m+1},q_{m+2},\dots)
  \end{align}
\end{subequations}
for the $m$-shifted sequences where $\bm{p}_{0} = \bm{p}$ and $\bm{q}_{0} = \bm{q}$.
We define the {\em generalized little $q$-Laguerre polynomial} of degree $n \in \mathbb{Z}_{\ge 0}$ by
\begin{align} \label{eq:GenLQLP}
  \mathcal{L}_{n}(x;a;\bm{p},\bm{q}) = {}
  \sum_{i=0}^{n} x^{i} \left( \prod_{k=i}^{n-1} \bm{p}^{\overline{k}} \right)
  \sum_{i \ge \nu_{i} \ge \cdots \ge \nu_{n-1} \ge 0} \prod_{k=i}^{n-1}
  \left( a \bm{q}^{\overline{k-\nu_k}} - \frac{1}{\bm{p}^{\overline{\nu_k}}} \right)
\end{align}
where the second sum in the right-hand side is over
all the $n-i$ non-increasing nonnegative integers $\nu_i,\dots,\nu_{n-1}$ at most $i$.

The generalized little $q$-Laguerre polynomials, as the name suggests, generalize
the little $q$-Laguerre polynomials as follows.

\begin{prop} \label{prop:GenLQLP2LQLP}
  If $p_{\ell} = q_{\ell} = q$ for every $\ell$ then 
  $\mathcal{L}_{n}(x;a;\bm{p},\bm{q}) = L_n(x;aq^{-1};q)$.
\end{prop}

\begin{proof}
  Suppose that $p_{\ell} = q_{\ell} = q$ for every $\ell$.
  We then have
  \begin{multline} \label{eq:GenLQLP2LQLP00}
    \mathcal{L}_{n}(x;a;\bm{p},\bm{q})
    = (-1)^{n} q^{\frac{n(n-1)}{2}} (a;q)_{n} \\
    \times \sum_{i=0}^{n} \frac{(-x)^{i} q^{-\frac{i(i-1)}{2}}}{(a;q)_{i}}
    \sum_{i \ge \nu_{i} \ge \cdots \ge \nu_{n-1} \ge 0} q^{-\sum_{k=i}^{n-1} \nu_k}.
  \end{multline}
  The second sum in the right-hand side reads
  \begin{align} \label{eq:GenLQLP2LQLPPPs}
    \sum_{i \ge \nu_{i} \ge \cdots \ge \nu_{n-1} \ge 0} q^{-\sum_{k=i}^{n-1} \nu_k}
    = \sum_{\pi \in \mathcal{P}(1,n-i,i)} q^{-|\pi|}
    = \frac{(-1)^{i} q^{\frac{i(i+1)}{2}} (q^{-n};q)_{i}}{(q;q)_{i}}
  \end{align}
  where we used \eqref{eq:NormGFBSP}.
  We get the result from
  \eqref{eq:LQLP}, \eqref{eq:GenLQLP2LQLP00} and \eqref{eq:GenLQLP2LQLPPPs}.
\end{proof}

Before stating the orthogonality of the generalized little $q$-Laguerre polynomials we show
a summation formula which will be used to prove the orthogonality.

\begin{lem} \label{lem:GenQCV}
  Let $n \in \mathbb{Z}_{\ge 0}$.
  Let $a$, $c$, $p_1,\dots,p_{n-1}$ and $q_1,\dots,q_{n-1}$ be indeterminates.
  Then
  \begin{align} \label{eq:GenQCV}
    \sum_{i=0}^{n} \left\{ \prod_{k=0}^{i-1} \left( \frac{1}{\bm{p}^{\overline{k}}} - a \right) \right\}
    \sum_{i \ge \nu_{i} \ge \cdots \ge \nu_{n-1} \ge 0} \prod_{k=i}^{n-1}
    \left( c \bm{q}^{\overline{k-\nu_k}} - \frac{1}{\bm{p}^{\overline{\nu_k}}} \right)
    = \prod_{k=0}^{n-1} (c \bm{q}^{\overline{k}} - a).
  \end{align}
\end{lem}

The proof of Lemma \ref{lem:GenQCV} is given in Appendix \ref{sec:GenQCV}.

The summation formula \eqref{eq:GenQCV} generalizes the $q$-Chu--Vandermonde identity \eqref{eq:QCV}
\cite[Eq.~(1.11.5)]{Koekoek-Leskey-Swarttow(2010)}, \cite[Theorem 12.2.4]{Ismail(2005CQOP)}.
In fact, the \eqref{eq:GenQCV} recovers \eqref{eq:QCV} with
the specialized parameters $p_{\ell} = q_{\ell} = q$ for every $\ell$.
(The method used to prove Proposition \ref{prop:GenLQLP2LQLP} is also applicable to see that.)

We now state the orthogonality of the generalized little $q$-Laguerre polynomials.
Let us fix the linear functional $\mathcal{F}: \mathbb{K}[x^{\pm 1},y^{\pm 1}] \to \mathbb{K}$,
$\mathbb{K} = \mathbb{C}(a,\bm{p},\bm{q})$, by the moments
\begin{subequations} \label{eq:GenLQLPMoments}
  \begin{align}
    f_{i,j} = \mathcal{F}[x^i y^j]
    = \prod_{k=1}^{\infty}
    \frac{1 - a \bm{p}^{\overline{i-k}} \bm{q}^{\overline{j}}}{1 - a \bm{p}^{\overline{-k}} \bm{q}^{\overline{j}}}
    & = \prod_{k=0}^{i-1} (1 - a \bm{p}^{\overline{k}} \bm{q}^{\overline{j}})
    && \text{if $i > 0$;} \\
    & = 1
    && \text{if $i = 0$;} \\
    & = \prod_{k=i}^{-1} (1 - a \bm{p}^{\overline{k}} \bm{q}^{\overline{j}})^{-1}
    && \text{if $i < 0$}
  \end{align}
\end{subequations}
for $(i,j) \in \mathbb{Z}^{2}$.
For $(r,c) \in \mathbb{Z}^{2}$ and $n \in \mathbb{Z}_{\ge 0}$ let
\begin{align} \label{eq:GenQLaguerreShifted}
  \mathcal{L}^{(r,c)}_{n}(x;a;\bm{p},\bm{q})
  = \mathcal{L}_{n}(x;a \bm{p}^{\overline{r}} \bm{q}^{\overline{c}}; \bm{p}_{r},\bm{q}_{c}).
\end{align}

\begin{thm} \label{thm:GenLQLPOrthty}
  The generalized little $q$-Laguerre polynomial $\mathcal{L}^{(r,c)}_{n}(y;a;\bm{p},\bm{q})$ satisfies
  the orthogonality \eqref{eq:BOPOrthty} with $P^{(r,c)}_{n}(x) = \mathcal{L}^{(r,c)}_{n}(x;a;\bm{p},\bm{q})$ and
  \begin{align} \label{eq:GenLQLPNormConst}
    h^{(r,c)}_{n} = f_{r,c+n} \times a^{n} \prod_{k=0}^{n-1} \bm{p}^{\overline{r+k}} (\bm{q}^{\overline{c+k}}  - \bm{q}^{\overline{c+n}})
  \end{align}
  with respect to the linear functional $\mathcal{F}$ having the moments \eqref{eq:GenLQLPMoments} where
  $f_{r,c+n} = \mathcal{F}[x^{r} y^{c+n}]$.
\end{thm}

\begin{proof}
  We have from \eqref{eq:GenLQLP}, \eqref{eq:GenLQLPMoments} and \eqref{eq:GenQLaguerreShifted} that
  \begin{multline} \label{eq:GenLQLPOrthty00}
    \frac{\mathcal{F}[x^{r} y^{c+j} \mathcal{L}^{(r,c)}_{n}(x;a;\bm{p},\bm{q})]}{f_{r,c+j} \prod_{k=0}^{n-1} \bm{p}_{r}^{\overline{k}}} \\
    = \sum_{i=0}^{n} \left\{ \prod_{k=0}^{i-1} \left( \frac{1}{ \bm{p}_{r}^{\overline{k}}} - a \bm{p}^{\overline{r}} \bm{q}^{\overline{c+j}} \right) \right\} \sum_{i \ge \nu_i \ge \cdots \ge \nu_{n-1} \ge 0} \prod_{k=i}^{n-1} \left( a  \bm{p}^{\overline{r}} \bm{q}^{\overline{c}} \bm{q}_{c}^{\overline{k-\nu_k}} - \frac{1}{\bm{p}_{r}^{\overline{\nu_k}}} \right).
  \end{multline}
  Lemma \ref{lem:GenQCV} with parameters
  \begin{align}
    a \gets a \bm{p}^{\overline{r}} \bm{q}^{\overline{c+j}}, \qquad
    c \gets a \bm{p}^{\overline{r}} \bm{q}^{\overline{c}}, \qquad
    p_{\ell} \gets p_{r+\ell}, \qquad
    q_{\ell} \gets q_{c+\ell},
  \end{align}
  such that $\bm{p} \gets \bm{p}_{r}$ and $\bm{q} \gets \bm{q}_{c}$, allows us to replace
  the right-hand side of \eqref{eq:GenLQLPOrthty00} with
  \begin{align}
    \prod_{k=0}^{n-1} (a \bm{p}^{\overline{r}} \bm{q}^{\overline{c}} \cdot \bm{q}_{c}^{\overline{k}} - a \bm{p}^{\overline{r}} \bm{q}^{\overline{c+j}})
    = a^{n} \prod_{k=0}^{n-1} \bm{p}^{\overline{r}} (\bm{q}^{\overline{c+k}} - \bm{q}^{\overline{c+j}}).
  \end{align}
  We thus have the orthogonality stated in the theorem since the last product vanishes for $0 \le j < n$ 
\end{proof}

The adjacent relations for the generalized little $q$-Laguerre polynomials are given as follows.

\begin{cor} \label{cor:GenLQLPARCfs}
  The generalized little $q$-Laguerre polynomials $\mathcal{L}^{(r,c)}_{n}(y;a;\bm{p},\bm{q})$ satisfy
  the adjacent relations \eqref{eq:ARs} with $P^{(r,c)}_{n}(x) = \mathcal{L}^{(r,c)}_{n}(x;a;\bm{p},\bm{q})$ and
  \begin{subequations} \label{eq:GenLQLPARCfs}
    \begin{align}
      a^{(r,c)}_{n} &= \bm{p}_{r}^{\overline{n}} (1 - a \bm{p}^{\overline{r}} \bm{q}^{\overline{c+n}}), \\
      b^{(r,c)}_{n} &= a \bm{p}^{\overline{r+n-1}} \bm{q}^{\overline{c}} (1 - \bm{q}_{c}^{\overline{n}}).
    \end{align}
  \end{subequations}
\end{cor}

\begin{proof}
  Proposition \ref{prop:ARs} and Theorem \ref{thm:GenLQLPOrthty} directly yield the result.
\end{proof}

Let us apply the lattice path combinatorics for biorthogonal polynomials in Section \ref{sec:LPs} to
the generalized little $q$-Laguerre polynomials.
Corollary \ref{cor:GenLQLPARCfs} suggests
the following labelling for edges of the square lattice $\mathbb{Z}_{\ge 0}^{2}$:
The vertical edge between $(i,j)$ and $(i-1,j)$ is labelled by
\begin{subequations} \label{eq:GenLQLPLabelling}
  \begin{align}
    \alpha_{i,j}
    & = \bm{p}_{i-j-1}^{\overline{j}} (1 - a \bm{p}^{\overline{i-j-1}} \bm{q}^{\overline{j}}) && \text{if $i > j$;} \\
    & = a \bm{p}^{\overline{i-1}} \bm{q}^{\overline{j-i}} (1 - \bm{q}_{j-i}^{\overline{i}})   && \text{if $i \le j$}
  \end{align}
\end{subequations}
while every horizontal edge by $1$.
We consider in the rest of this section and in Section \ref{sec:NFPPBSP02}
the weights of lattice paths with respect to this labelling.

Let $\lambda$ be a partition.
For each $\ell \in \mathbb{Z}$ we define $\mathsf{D}_{\ell}(\lambda)$ to be
the number of boxes on the $\ell$-th diagonal of the Young diagram of $\lambda$ where
a box at $(i,j)$ is said to be on the $\ell$-th diagonal if and only if $j-i = \ell$.
Especially $\mathsf{D}_{0}(\lambda) = \mathsf{D}(\lambda)$ that measures
the size of the Durfee square of the Young diagram of $\lambda$.
For example, the Young diagram $\lambda = \lambda(P)$ of
the lattice path $P$ in Figure \ref{fig:SqLattice} satisfies that
$(\mathsf{D}_{\ell}(\lambda))_{-4 < \ell < 6} = (1,2,2,3,3,2,1,1,0)$.

Let $(r,c) \in \mathbb{Z}_{\ge 0}^{2}$ and
let $P$ be a lattice path on the square lattice $\mathbb{Z}_{\ge 0}^{2}$ going from $(r,0)$ to $(0,c)$.
If $\ell \ge 0$ (resp.~if $\ell < 0$),
$\mathsf{D}_{\ell}(\lambda(P)) = d$ if and only if $P$ passes through $(d,d+\ell)$ (resp.~through $(d-\ell,d)$).
We write $\lambda_i(P)$ for the $i$-th part of the partition $\lambda(P)$.

\begin{lem} \label{lem:GenLQLPWeight}
  Let $(r,c) \in \mathbb{Z}_{\ge 0}^{2}$ and
  let $P$ be a lattice path on the square lattice $\mathbb{Z}_{\ge 0}^{2}$ going from $(r,0)$ to $(0,c)$.
  The weight $w(P)$ with respect to the labelling given by \eqref{eq:GenLQLPLabelling} then admits that
  \begin{subequations} \label{eq:GenLQLPWeight}
    \begin{align}
      \label{eq:GenLQLPWeightW}
      &
      w(P) =
      a^{\mathsf{D}_{0}(\lambda(P))}
      \left( \prod_{i=1}^{r-1} p_i^{\mathsf{D}_{-i}(\lambda(P))} \right)
      \left( \prod_{j=1}^{c-1} q_j^{\mathsf{D}_{j}(\lambda(P))} \right)
      \omega'_{r}(P;a;\bm{p},\bm{q}) \qquad \text{where} \\
      \label{eq:GenLQLPWeightOmega'}
      &
      \omega'_{r}(P;a;\bm{p},\bm{q}) =
      \left\{ \prod_{i=1}^{d} (1 - \bm{q}_{\lambda_{i}(P)-i}^{\overline{i}}) \right\}
      \left\{ \prod_{i=d+1}^{r} (1 - a \bm{p}^{\overline{i-\lambda_{i}(P)-1}} \bm{q}^{\overline{\lambda_{i}(P)}}) \right\}
    \end{align}
  \end{subequations}
  where $d = \mathsf{D}(\lambda(P))$.
\end{lem}

\begin{proof}
  The proof is totally parallel to that of Lemma \ref{lem:LQLPPathWeight} in Section \ref{sec:LQLP}.
  We ``factor'' the labelling \eqref{eq:GenLQLPLabelling} into two distinct labellings;
  one puts on the vertical edge between $(i,j)$ and $(i-1,j)$ the label
  \begin{subequations} \label{eq:GenLQLPLabelsFactored01}
    \begin{align}
      \alpha'_{i,j}
      & = \bm{p}_{i-j-1}^{\overline{j}}                     && \text{if $i > j$;} \\
      & = a \bm{p}^{\overline{i-1}} \bm{q}^{\overline{j-i}} && \text{if $i \le j$,}
    \end{align}
  \end{subequations}
  and the other
  \begin{subequations} \label{eq:GenLQLPLabelsFactored02}
    \begin{align}
      \alpha''_{i,j}
      & = (1 - a \bm{p}^{\overline{i-j-1}} \bm{q}^{\overline{j}}) && \text{if $i > j$;} \\
      & = (1 - \bm{q}_{j-i}^{\overline{i}})                       && \text{if $i \le j$,}
    \end{align}
  \end{subequations}
  where both the labellings put $1$ on every horizontal edge.
  For any lattice path $P$ we write
  $w'(P)$ and $w''(P)$ for the weights of $P$ with respect to the labellings
  \eqref{eq:GenLQLPLabelsFactored01} and \eqref{eq:GenLQLPLabelsFactored02} respectively.
  Obviously $w(P) = w'(P) w''(P)$.

  Let $P$ be a lattie path going from $(r,0)$ to $(0,c)$ and let
  $d = \mathsf{D}(\lambda(P)) = \mathsf{D}_{0}(\lambda(P))$.
  We then have from \eqref{eq:GenLQLPLabelsFactored01} that
  \begin{align}
    w'(P) = \left( \prod_{i=1}^{d} a \bm{p}^{\overline{i-1}} \bm{q}^{\overline{\lambda_{i}(P)-i}} \right)
    \left( \prod_{i=d+1}^{r} \bm{p}_{i-\lambda_{i}(P)-1}^{\overline{\lambda_{i}(P)}} \right).
  \end{align}
  since $P$ passes through the vertical edge between
  $(i,\lambda_i(P))$ and $(i-1,\lambda_i(P))$ for each integer $i$, $1 \le i \le r$.
  The last expression is actually equivalent to
  \begin{align} \label{eq:GenLQLPWeight01}
    w'(P) = a^{\mathsf{D}_{0}(\lambda(P))}
    \left( \prod_{i=1}^{r-1} p_i^{\mathsf{D}_{-i}(\lambda(P))} \right)
    \left( \prod_{j=1}^{c-1} q_j^{\mathsf{D}_{j}(\lambda(P))} \right).
  \end{align}
  To see this we consider to fill in the Young diagram $\lambda(P)$ by
  writing in to the $\lambda_i(P)$ boxes in the $i$-th row 
  \begin{subequations}
    \begin{align}
      & p_1,~ \dots,~ p_{i-1},~ a,~ q_1,~ \dots,~ q_{\lambda_i(P)-i} && \text{if $1 \le i \le d$;} \\
      & p_{i-\lambda_i(P)-1},~ \dots,~ p_{i-2}                       && \text{if $d < i \le r$}
    \end{align}
  \end{subequations}
  from left to right.
  For example, the lattice path in Figure \ref{fig:SqLattice} gives rise to
  the filling shown in Figure \ref{fig:Filling}.
  \begin{figure}
    \centering
    \includegraphics{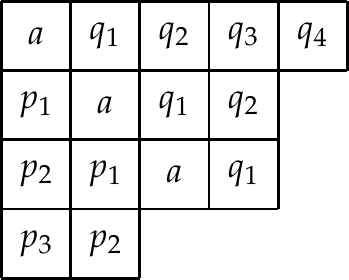}
    \caption{%
      The filling of the Young diagram of the lattice path in Figure \ref{fig:SqLattice}.
    }
    \label{fig:Filling}
  \end{figure}
  The way of filling ensures that the product of all the entries in the Young diagram is equal to $w'(P)$.
  In addition the entries $a$, $p_i$ and $q_j$ reside in the boxes on the main, $(-i)$-th and $j$-th diagonals
  respectively.
  We therefore have \eqref{eq:GenLQLPWeight01}.
  It is easy to find from \eqref{eq:GenLQLPLabelsFactored02} that
  $w''(P)$ is equal to $\omega'_{r}(P;a;\bm{p},\bm{q})$ defined by \eqref{eq:GenLQLPWeightOmega'}.
  We thus have \eqref{eq:GenLQLPWeightW} since $w(P) = w'(P) w''(P)$.
\end{proof}

Theorem \ref{thm:MomentsLPaths} and Lemma \ref{lem:GenLQLPWeight} imply the following.

\begin{thm} \label{thm:GenLQLPMomentsLPaths}
  Let $\mathcal{F}$ be the linear functional (for the generalized little $q$-Laguerre polynomials) determined by
  the moments \eqref{eq:GenLQLPMoments}.
  Let $(r,c) \in \mathbb{Z}_{\ge 0}^{2}$.
  Then
  \begin{align}
    f_{r,c} = \sum_{P}
    a^{\mathsf{D}_{0}(\lambda(P))}
    \left( \prod_{i=1}^{r-1} p_i^{\mathsf{D}_{-i}(\lambda(P))} \right)
    \left( \prod_{j=1}^{c-1} q_j^{\mathsf{D}_{j}(\lambda(P))} \right)
    \omega'_{r}(P;a;\bm{p},\bm{q})
  \end{align}
  where the sum ranges over
  all the lattice paths on the square lattice $\mathbb{Z}_{\ge 0}^{2}$ going from $(r,0)$ to $(0,c)$, and
  $\omega'_{r}(P;a;\bm{p},\bm{q})$ is defined by \eqref{eq:GenLQLPWeightOmega'}.
\end{thm}

\begin{proof}
  Delete $w(P)$ from \eqref{eq:MomentsLPaths} and \eqref{eq:GenLQLPWeight} where
  $f_{0,c} = 1$ from \eqref{eq:GenLQLPMoments}.
\end{proof}

Corollary \ref{cor:DetNILPs} and Lemma \ref{lem:GenLQLPWeight} imply the following.

\begin{thm} \label{thm:GenLQLPDetsLPaths}
  Let $\mathcal{F}$ be the linear functional (for the generalized little $q$-Laguerre polynomials) determined by
  the moments \eqref{eq:GenLQLPMoments}.
  Let $(r,c,n) \in \mathbb{Z}_{\ge 0}^{3}$.
  Then
  \begin{subequations} \label{eq:GenLQLPDetLPaths}
    \begin{multline}
      \Delta^{(r,c)}_{n}
      \left[ \left\{ \prod_{k=0}^{n-1} (p_{r+k} q_{c+k})^{\frac{(n-k)(n-k-1)}{2}} \right\}
        \left\{ \prod_{1 \le i \le k < n} (1 - \bm{q}_{c+k-i}^{\overline{i}}) \right\} \right]^{-1} \\
      = \sum_{(P_0,\dots,P_{n-1}) \in \mathcal{LP}(r,c,n)}
      a^{\sum_{k=0}^{n-1} \mathsf{D}_{0}(\lambda(P_k))}
      \left( \prod_{i=1}^{r-1} p_{i}^{\sum_{k=0}^{n-1} \mathsf{D}_{-i}(\lambda(P_k))} \right)
      \left( \prod_{j=1}^{c-1} q_{j}^{\sum_{k=0}^{n-1} \mathsf{D}_{j}(\lambda(P_k))} \right) \\
      \times \omega'_{r,n}(P_0,\dots,P_{n-1};a;\bm{p},\bm{q})
    \end{multline}
    where
    \begin{multline}
      \omega'_{r,n}(P_0,\dots,P_{n-1};a;\bm{p},\bm{q}) \\
      = \prod_{k=0}^{n-1} \left\{ \prod_{i=k+1}^{d_k} (1 - \bm{q}_{\lambda_i(P_k)-i}^{\overline{i}}) \right\}
      \left\{ \prod_{i=d_k+1}^{r+k} (1 - a \bm{p}^{\overline{i-\lambda_i(P_k)-1}} \bm{q}^{\overline{\lambda_i(P_k)}}) \right\}.
    \end{multline}
  \end{subequations}
  where $d_k = \mathsf{D}(\lambda(P_k))$.
\end{thm}

\begin{proof}
  For a lattice path $P$ let $w'(P)$ and $w''(P)$ be the weights of $P$ defined in
  the proof of Lemma \ref{lem:GenLQLPWeight}.
  Then $w'(P)$ satisfies \eqref{eq:GenLQLPWeight01}, and
  $w''(P)$ is equal to $\omega'_{r}(P;a;\bm{p},\bm{q})$ defined by \eqref{eq:GenLQLPWeightOmega'}.
  Since $w(P) = w'(P) w''(P)$ Corollary \ref{cor:DetNILPs} and Lemma \ref{lem:GenLQLPWeight} induce that
  \begin{align} \label{eq:GenLQLPDetWeight00}
    \Delta^{(r,c)}_{n} = \sum_{(P_0,\dots,P_{n-1}) \in \mathcal{LP}(r,c,n)} \prod_{k=0}^{n-1} w'(P_k) w''(P_k)
  \end{align}
  where $f_{0,c+k} = 1$ for every $k$ from \eqref{eq:GenLQLPMoments}.

  Let $(P_0,\dots,P_{n-1}) \in \mathcal{LP}(r,c,n)$.
  The condition for the lattice paths $P_0,\dots,P_{n-1}$ to be non-intersecting forces $P_k$ to
  start (from $(r+k,0)$) with $k$ consecutive east steps and ends (to $(0,c+k)$) with $k$ consecutive north steps.
  Hence
  \begin{subequations}
    \begin{align}
      \label{eq:DForced}
      & \mathsf{D}_{-r-i}(P_k) = \mathsf{D}_{c+i}(P_k) = k-i && \text{for $0 \le i \le k < n$;} \\
      \label{eq:LambdaForced}
      & \lambda_i(P_k) = c+k                                 && \text{for $1 \le i \le k < n$.}
    \end{align}
  \end{subequations}
  We have from \eqref{eq:DForced} that
  \begin{subequations} \label{eq:GenLQLPDetWeight01}
    \begin{align}
      \prod_{k=0}^{n-1} w'(P_k)
      & = \prod_{k=0}^{n-1} a^{\mathsf{D}_{0}(\lambda(P_k))}
      \left( \prod_{i=1}^{r+k-1} p_{i}^{\mathsf{D}_{-i}(\lambda(P_k))} \right)
      \left( \prod_{j=1}^{c+k-1} q_{j}^{\mathsf{D}_{j}(\lambda(P_k))} \right) \notag \\
      & = a^{\sum_{k=0}^{n-1} \mathsf{D}_{0}(\lambda(P_k))}
      \left( \prod_{i=1}^{r-1} p_{i}^{\sum_{k=0}^{n-1} \mathsf{D}_{-i}(\lambda(P_k))} \right)
      \left( \prod_{j=1}^{c-1} q_{j}^{\sum_{k=0}^{n-1} \mathsf{D}_{j}(\lambda(P_k))} \right) \notag \\
      & \qquad \times \left\{ \prod_{0 \le i \le k < n} (p_{r+i} q_{c+i})^{k-i} \right\}
    \end{align}
    where
    \begin{align}
      \prod_{0 \le i \le k < n} (p_{r+i} q_{c+i})^{k-i}
      = \prod_{k=0}^{n-1} (p_{r+k} q_{c+k})^{\frac{(n-k)(n-k-1)}{2}}.
    \end{align}
  \end{subequations}
  We have from \eqref{eq:LambdaForced} that
  \begin{align} \label{eq:GenLQLPDetWeight02}
    \prod_{k=0}^{n-1} w''(P_k)
    & = \prod_{k=0}^{n-1} \left\{ \prod_{i=1}^{d_k} (1 - \bm{q}_{\lambda_i(P_k)-i}^{\overline{i}}) \right\}
    \left\{ \prod_{i=d_k+1}^{r+k} (1 - a \bm{p}^{\overline{i-\lambda_i(P_k)-1}} \bm{q}^{\overline{\lambda_i(P_k)}}) \right\} \notag \\
    & = \omega'_{r,n}(P_0,\dots,P_{n-1};a;\bm{p},\bm{q}) \times \prod_{1 \le i \le k < n} (1 - \bm{q}_{c+k-i}^{\overline{i}})
  \end{align}
  Substituting \eqref{eq:GenLQLPDetWeight01} and \eqref{eq:GenLQLPDetWeight02} for
  \eqref{eq:GenLQLPDetWeight00} we obtain \eqref{eq:GenLQLPDetLPaths} in the theorem.
\end{proof}

Theorem \ref{thm:GenLQLPOrthty}, Corollary \ref{cor:GenLQLPARCfs}, Lemma \ref{lem:GenLQLPWeight} and
Theorems \ref{thm:GenLQLPMomentsLPaths} and \ref{thm:GenLQLPDetsLPaths} in this section respectively
recover Proposition \ref{prop:LQLPOrthty}, Corollary \ref{cor:LQLPARs},
Lemma \ref{lem:LQLPPathWeight} and Theorems \ref{thm:LQLPMomentsLPaths} and \ref{thm:LQLPDetNILP}
in Section \ref{sec:LQLP}
with the specialized parameters $a \gets a q$ and $p_i = q_j = q$.
This reduction is consistent with that of the generalized little $q$-Laguerre polynomials to
the little $q$-Laguerre polynomials mentioned in Proposition \ref{prop:GenLQLP2LQLP}.

\section{Nice formula for plane partitions with bounded size of parts, II}
\label{sec:NFPPBSP02}

We derive in Section \ref{sec:NFPPBSP02} another nice formula for plane partitions with bounded size of parts
based on the generalized little $q$-Laguerre polynomials introduced and examined in Section \ref{sec:GenLQLP}.
The nice formula would generalize
the {\em trace generating function} for plane partitions with {\em unbounded size of parts}
\begin{align} \label{eq:LTrGFRep}
  \sum_{\pi \in \mathcal{P}(r,c)} \prod_{-r < \ell < c} q_{\ell}^{\mathsf{tr}_{\ell}(\pi)}
  = \prod_{i=0}^{r-1} \prod_{j=0}^{c-1} \left( 1 - \prod_{\ell=-i}^{j} q_{\ell} \right)^{-1}
\end{align}
developed by Gansner \cite{Gansner(1981Burge),Gansner(1981HG)} where
$\mathsf{tr}_{\ell}(\pi)$ denotes the {\em $\ell$-trace} of a plane partition $\pi = (\pi_{i,j})$ defined by
$\mathsf{tr}_{\ell}(\pi) = \sum_{j-i=\ell} \pi_{i,j}$.

The discussion in this section is totally parallel to that in Section \ref{sec:NFPPBSP01}:
Employ the bijection between $\mathcal{P}(r,c,n)$ and $\mathcal{LP}(r,c,n)$ to translate
Theorem \ref{thm:GenLQLPDetsLPaths} in the language of plane partitions.
(The bijection is discussed in Section \ref{sec:NFPPBSP01}.)

Let us remind several symbols defined in the preceding sections.
For a lattice path $P$ on the square lattice $\mathbb{Z}_{\ge 0}^{2}$
$\lambda(P)$ denotes the (integer) partition whose Young diagram is given by the finite region bordered by $P$;
for a plane partition $\pi$, $\lambda_k(\pi)$ the partition whose Young diagram is given by
the cross-section at level $k$ of the 3D Young diagram of $\pi$, see Figure \ref{fig:3DYDCrossSec}.
Let us write $\lambda_{k,i}(\pi)$ for the $i$-th part of the partition $\lambda_k(\pi)$.

Suppose that a plane partition $\pi \in \mathcal{P}(r,c,n)$ and
an $n$-tuple $(P_0,\dots,P_{n-1}) \in \mathcal{LP}(r,c,n)$ of non-intersecting lattice paths
correspond to each other by the bijection.
\begin{subequations} \label{eq:BjStats02}
  It readily follows from \eqref{eq:Bijection} that
  \begin{align}
    \mathsf{D}_{\ell}(\lambda(P_k)) & = \mathsf{D}_{\ell}(\lambda_{n-k}(\pi)) + k
  \end{align}
  for $0 \le k < n$ and $-r < \ell < c$ that implies
  \begin{align} \label{eq:BijectionTrL}
    \sum_{k=0}^{n-1} \mathsf{D}_{\ell}(\lambda(P_k)) & = \mathsf{tr}_{\ell}(\pi) + \frac{n(n-1)}{2}.
  \end{align}
  It also follows from \eqref{eq:Bijection} that
  \begin{align} \label{eq:BijectionLambdaI}
    \lambda_{i}(P_k) = \lambda_{n-k,i-k}(\pi) + k
  \end{align}
  for $0 \le k < n$ and $k < i \le r+k$.
\end{subequations}

\begin{thm} \label{thm:NFPPBSP02}
  Let $(r,c,n) \in \mathbb{Z}_{\ge 0}^{3}$.
  \begin{subequations} \label{eq:NFPPBSP02}
    Then
    \begin{multline} \label{eq:NFPPBSP02NF}
      \sum_{\pi \in \mathcal{P}(r,c,n)}
      a^{\mathsf{tr}_{0}(\pi)}
      \left( \prod_{i=1}^{r-1} p_{i}^{\mathsf{tr}_{-i}(\pi)} \right)
      \left( \prod_{j=1}^{c-1} q_{j}^{\mathsf{tr}_{j}(\pi)} \right)
      \omega_{r,n}(\pi;a;\bm{p},\bm{q}) \\
      = \prod_{i=0}^{r-1} \prod_{j=0}^{c-1} \prod_{k=0}^{n-1}
      \frac{1 - a \bm{p}^{\overline{i}} \bm{q}^{\overline{j+k+1}}}{1 - a \bm{p}^{\overline{i}} \bm{q}^{\overline{j+k}}}
    \end{multline}
    where
    \begin{multline}
      \omega_{r,n}(\pi;a;\bm{p},\bm{q}) = \prod_{k=1}^{\pi_{1,1}}
      \left\{ \prod_{i=1}^{D_k} (1 - \bm{q}_{\lambda_{k,i}(\pi)-i}^{\overline{n-k+i}}) \right\} \\
      \times \left\{ \prod_{i=D_k+1}^{r} (1 - a \bm{p}^{\overline{i-\lambda_{k,i}(\pi)-1}} \bm{q}^{\overline{n-k+\lambda_{k,i}(\pi)}}) \right\} \left\{ \prod_{i=1}^{r} (1 - a \bm{p}^{\overline{i-1}} \bm{q}^{\overline{n-k}}) \right\}^{-1}
    \end{multline}
    where $\pi_{1,1}$ denotes the $(1,1)$-part of a plane partition $\pi$, and
    $D_k = \mathsf{D}(\lambda_k(\pi))$.
  \end{subequations}
\end{thm}

\begin{proof}
  Suppose that $\pi \in \mathcal{P}(r,c,n)$ and $(P_0,\dots,P_{n-1}) \in \mathcal{LP}(r,c,n)$ correspond to
  each other by the bijection.
  We then have from \eqref{eq:BijectionTrL} that
  \begin{multline}
    a^{\sum_{k=0}^{n-1} \mathsf{D}_{0}(\lambda(P_k))}
    \left( \prod_{i=1}^{r-1} p_{i}^{\sum_{k=0}^{n-1} \mathsf{D}_{-i}(\lambda(P_k))} \right)
    \left( \prod_{j=1}^{c-1} q_{j}^{\sum_{k=0}^{n-1} \mathsf{D}_{j}(\lambda(P_k))} \right) \\
    = \left\{ a \left( \prod_{i=1}^{r-1} p_i \right) \left( \prod_{j=1}^{c-1} q_j \right) \right\}^{\frac{n(n-1)}{2}}
    \times a^{\mathsf{tr}_{0}(\pi)}
    \left( \prod_{i=1}^{r-1} p_{i}^{\mathsf{tr}_{-i}(\pi)} \right)
    \left( \prod_{j=1}^{c-1} q_{j}^{\mathsf{tr}_{j}(\pi)} \right).
  \end{multline}
  We also have from \eqref{eq:BijectionLambdaI} that
  \begin{align}
    \omega'_{r,n}(P_0,\dots,P_{n-1};a;\bm{p},\bm{q})
    = \left\{ \prod_{i=0}^{r-1} \prod_{k=0}^{n-1} (1 - a \bm{p}^{\overline{i}} \bm{q}^{\overline{k}}) \right\}
    \times \omega_{r,n}(\pi;a;\bm{p},\bm{q})
  \end{align}
  Note that $\lambda_{k}(\pi) = \emptyset$ and $\lambda_{k,i}(\pi) = 0$ for $k > \pi_{1,1}$.
  The formula \eqref{eq:GenLQLPDetLPaths} in Theorem \ref{thm:GenLQLPDetsLPaths} is hence equivalent to
  \begin{subequations}
    \begin{multline} \label{eq:NFPPBSP0200}
      \sum_{\pi \in \mathcal{P}(r,c,n)}
      a^{\mathsf{tr}_{0}(\pi)}
      \left( \prod_{i=1}^{r-1} p_{i}^{\mathsf{tr}_{-i}(\pi)} \right)
      \left( \prod_{j=1}^{c-1} q_{j}^{\mathsf{tr}_{j}(\pi)} \right)
      \omega_{r,n}(\pi;a;\bm{p},\bm{q}) \\
      = \frac{\Delta^{(r,c)}_{n}}{\kappa^{(r,c)}_{n}}
      \left\{ \prod_{i=0}^{r-1} \prod_{k=0}^{n-1} (1 - a \bm{p}^{\overline{i}} \bm{q}^{\overline{k}}) \right\}^{-1}
    \end{multline}
    where
    \begin{multline}
      \kappa^{(r,c)}_{n}
      = \left\{ a \left( \prod_{i=1}^{r-1} p_i \right) \left( \prod_{j=1}^{c-1} q_j \right) \right\}^{\frac{n(n-1)}{2}}
      \left\{ \prod_{k=0}^{n-1} (p_{r+k} q_{c+k})^{\frac{(n-k)(n-k-1)}{2}} \right\} \\
      \times \left\{ \prod_{1 \le i \le k < n} (1 - \bm{q}_{c+k-i}^{\overline{i}}) \right\}.
    \end{multline}
  \end{subequations}
  The proof thus amounts to the evaluation of the determinant $\Delta^{(r,c)}_{n}$ of
  moments \eqref{eq:GenLQLPMoments} of the generalized little $q$-Laguerre polynomials
  (examined in Section \ref{sec:GenLQLP}).
  We have from \eqref{eq:BOPNormConstDet} that
  \begin{align}
    \Delta^{(r,c)}_{n} = \prod_{k=0}^{n-1} h^{(r,c)}_{k}
  \end{align}
  for general biorthogonal polynomials.
  Substituting the normalization constant \eqref{eq:GenLQLPNormConst} of
  the generalized little $q$-Laguerre polynomials for the right-hand side we straightforwardly find that
  \begin{align} \label{eq:GenLQLPDetEval}
    \Delta^{(r,c)}_{n} = \kappa^{(r,c)}_{n}
    \prod_{i=0}^{r-1} \prod_{k=0}^{n-1} (1 - a \bm{p}^{\overline{i}} \bm{q}^{\overline{c+k}}).
  \end{align}
  Substituting \eqref{eq:GenLQLPDetEval} for \eqref{eq:NFPPBSP0200} we obtain the formula \eqref{eq:NFPPBSP02NF}.
\end{proof}

The nice formula \eqref{eq:NFPPBSP02} for plane partitions with bounded size of parts generalizes
the trace generating function \eqref{eq:LTrGFRep} for those with unbounded size of parts.
Indeed, $\mathcal{P}(r,c,n) \to \mathcal{P}(r,c)$, $\omega_{r,n}(\pi;a;\bm{p},\bm{q}) \to 1$ and
\begin{align} \label{eq:TrGFRefinedReduction}
  \prod_{i=0}^{r-1} \prod_{j=0}^{c-1} \prod_{k=0}^{n-1} \frac{1 - a \bm{p}^{\overline{i}} \bm{q}^{\overline{j+k+1}}}{1 - a \bm{p}^{\overline{i}} \bm{q}^{\overline{j+k}}}
  = \prod_{i=0}^{r-1} \prod_{j=0}^{c-1} \frac{1 - a \bm{p}^{\overline{i}} \bm{q}^{\overline{n+j}}}{1 - a \bm{p}^{\overline{i}} \bm{q}^{\overline{j}}}
  \to \prod_{i=0}^{r-1} \prod_{j=0}^{c-1} (1 - a \bm{p}^{\overline{i}} \bm{q}^{\overline{j}})^{-1}
\end{align}
as $n \to \infty$ since $\lim_{n \to \infty} \bm{q}^{\overline{n}} = 0$, where
the convergences of $\omega_{r,n}(\pi;a;\bm{p},\bm{q})$ and in \eqref{eq:TrGFRefinedReduction} are as
formal power series in $q_1,q_2,q_3,\dots$
(or as complex numbers with $|q_{\ell}| < \varepsilon < 1$ for every $\ell$ with some real $\varepsilon \in (0,1)$).
We thus obtain as a consequence of \eqref{eq:NFPPBSP02} that
\begin{align}
  \sum_{\pi \in \mathcal{P}(r,c)}
  a^{\mathsf{tr}_{0}(\pi)}
  \left( \prod_{i=1}^{r-1} p_{i}^{\mathsf{tr}_{-i}(\pi)} \right)
  \left( \prod_{j=1}^{c-1} q_{j}^{\mathsf{tr}_{j}(\pi)} \right)
  = \prod_{i=0}^{r-1} \prod_{j=0}^{c-1} (1 - a \bm{p}^{\overline{i}} \bm{q}^{\overline{j}})^{-1}
\end{align}
that is nothing but the trace generating function \eqref{eq:LTrGFRep} with
$a \gets q_0$ and $q_{-i} = p_{i}$ for $i \ge 1$.

The nice formula \eqref{eq:NFPPBSP02} also generalizes the nice formula \eqref{eq:NFPPBSO01} derived from
the little $q$-Laguerre polynomials where the former respects all the $\ell$-traces while
the latter only the ($0$-)trace and the norm that is equal to the sum of the $\ell$-traces.
It is easy to see that we can derive \eqref{eq:NFPPBSO01} from \eqref{eq:NFPPBSP02} by
the specialization that $a \gets a q$ and $p_i = q_j = q$ for every $i$ and $j$.
This reduction is consistent with that from the generalized little $q$-Laguerre polynomials to
the little $q$-Laguerre polynomials (Proposition \ref{prop:GenLQLP2LQLP}).


\appendix

\section{Proof of Lemma \ref{lem:GenQCV}}
\label{sec:GenQCV}

We give a proof of Lemma \ref{lem:GenQCV} in Section \ref{sec:GenLQLP}.
The proof depends on the following two facts.

\begin{fact} \label{fact:NewtonInterpol}
  Let $\alpha_0,\alpha_1,\alpha_2,\dots$ be a sequence of constants.
  Let us consider a (Newton) polynomial in $x$
  \begin{align}
    f(x) = \sum_{i=0}^{n} c_i \prod_{k=0}^{i-1} (x - \alpha_k)
  \end{align}
  with constant coefficients $c_i$.
  We determine constants $c^{(t)}_{i}$, $t \in \mathbb{Z}_{\ge 0}$, $0 \le i \le n$, by the recurrence
  \begin{align} \label{eq:NewtonInterpolDD}
    c^{(t+1)}_{i} = c^{(t)}_{i} + (\alpha_{t+i+1} - \alpha_{t}) c^{(t)}_{i+1}
  \end{align}
  with $c^{(0)}_{i} = c_i$ and $c^{(t)}_{n+1} = 0$.
  Then
  \begin{align}
    f(x) = \sum_{i=0}^{n} c^{(t)}_i \prod_{k=0}^{i-1} (x - \alpha_{t+k})
  \end{align}
  and therefore $f(\alpha_{t}) = c^{(t)}_{0}$ for each $t \in \mathbb{Z}_{\ge 0}$.
\end{fact}

The simple induction for $t \in \mathbb{Z}_{\ge 0}$ readily proves Lemma \ref{fact:NewtonInterpol}.
The statement of Lemma \ref{fact:NewtonInterpol} is nothing but interpolation by Newton polynomials where
the recurrence \eqref{eq:NewtonInterpolDD} reads the well-known {\em divided-difference}
\begin{align}
  c^{(t)}_{i} = \frac{c^{(t+1)}_{i-1} - c^{(t)}_{i-1}}{\alpha_{t+i} - \alpha_{t}},
\end{align}
see, e.g., \cite[\S 7.1]{Baker-GravesMorris(1996)}.

\begin{fact} \label{fact:Lem4GenQCV}
  Let $u_0,u_1,u_2,\dots$ and $v_1,v_2,v_3,\dots$ be sequences of constants.
  Let
  \begin{align} \label{eq:F4GenQCV}
    F(m,n,\tau)
    = \sum_{0 \le j_{1} \le \cdots \le j_{m} \le n} \prod_{k=1}^{m} (u_{j_{k}+\tau} - v_{j_{k}+k})
  \end{align}
  for $m,n,\tau \in \mathbb{Z}_{\ge 0}$ where $F(0,n,\tau) \equiv 1$.
  The recurrence
  \begin{align} \label{eq:Rec4GenQCV}
    F(m,n,\tau+1) = F(m,n,\tau) + (u_{n+1} - u_0) F(m-1,n+1,\tau)
  \end{align}
  then holds for $m \in \mathbb{Z}_{\ge 1}$ and $n,\tau \in \mathbb{Z}_{\ge 0}$.
\end{fact}

\begin{proof}
  We prove \eqref{eq:Rec4GenQCV} by induction with respect to $m = 1,2,3,\dots$.
  We have from \eqref{eq:F4GenQCV} that
  \begin{align}
    f(1,n,\tau) = \sum_{i=0}^{n} (u_{i+\tau} - v_{i+1})
  \end{align}
  that implies \eqref{eq:Rec4GenQCV} with $m=1$.
  Assume that $m \ge 2$ and, without any loss of generality, that $\tau = 0$.
  We then have from \eqref{eq:F4GenQCV} and the assumption of induction that
  \begin{subequations}
    \begin{align}
      F(m,n,1)
      = & \sum_{j_{m}=0}^{n} (u_{j_m+1} - v_{j_m+m}) F(m-1,j_m,1) \notag \\
      \label{eq:Rec4GenQCVSum01} 
      = & \sum_{j_{m}=0}^{n} (u_{j_m+1} - v_{j_m+m}) F(m-1,j_m,0) \\
      \label{eq:Rec4GenQCVSum02}
      + & \sum_{j_{m}=0}^{n} (u_{j_m+1} - v_{j_m+m}) (u_{j_m+1} - u_0) F(m-2,j_m+1,0).
    \end{align}
  \end{subequations}
  \begin{subequations}
    The sum in \eqref{eq:Rec4GenQCVSum01} is equal to
    \begin{multline} \label{eq:Rec4GenQCVSum01Mod}
      F(m,n,0) + \sum_{j_{m}=0}^{n} (u_{j_m+1} - u_{j_m}) F(m-1,j_m,0) \\
      = F(m,n,0) + \sum_{0 \le j_1 \le \cdots \le j_{m-1} \le n+1} (u_{n+1} - u_{j_{m-1}}) \prod_{k=1}^{m-1} (u_{j_{k}} - v_{j_{k}+k})
    \end{multline}
    while the sum in \eqref{eq:Rec4GenQCVSum02}
    \begin{multline} \label{eq:Rec4GenQCVSum02Mod}
      \sum_{j_{m}=1}^{n+1} (u_{j_m} - v_{j_m+m-1}) (u_{j_m} - u_0) F(m-2,j_m,0) \\
      = \sum_{0 \le j_1 \le \cdots \le j_{m-1} \le n+1} (u_{j_{m-1}} - u_0) \prod_{k=1}^{m-1} (u_{j_{k}} - v_{j_{k}+k}).
    \end{multline}
  \end{subequations}
  Gathering \eqref{eq:Rec4GenQCVSum01Mod} and \eqref{eq:Rec4GenQCVSum02Mod}
  we obtain the right-hand side of \eqref{eq:Rec4GenQCV}.
  That completes the proof.
\end{proof}

We now prove Lemma \ref{lem:GenQCV}.

\begin{proof}[Proof of Lemma \ref{lem:GenQCV}]
  The summation formula \eqref{eq:GenQCV} is equivalently written as follows:
  \begin{align} \label{eq:GenQCVMod}
    \sum_{i=0}^{n} \left\{ \prod_{k=0}^{i-1} \left( a - \frac{1}{\bm{p}^{\overline{k}}} \right) \right\}
    \sum_{0 \le j_1 \le \cdots \le j_{n-i} \le i} \prod_{k=1}^{n-i}
    \left( \frac{1}{\bm{p}^{\overline{j_k}}} - c \bm{q}^{\overline{n-j_k-k}} \right)
    = \prod_{k=0}^{n-1} (a - c \bm{q}^{\overline{k}}).
  \end{align}
  We think of the both sides of \eqref{eq:GenQCVMod} as polynomials in $a$, and
  write $f(a)$ and $g(a)$ for the left-hand and right-hand sides respectively.
  We prove $f(a) \equiv g(a)$ as polynomials by showing that $f(a) = g(a)$ for infinitely many $a$'s.
  It is obvious that
  \begin{align} \label{eq:qCVGenGVal}
    g(1/\bm{p}^{\overline{t}})
    {} = \prod_{k=0}^{n-1} \left( \frac{1}{\bm{p}^{\overline{t}}} - c \bm{q}^{\overline{k}} \right)
  \end{align}
  for any $t \in \mathbb{Z}$.
  We evaluate $f(1/\bm{p}^{\overline{t}})$ by use of Facts \ref{fact:NewtonInterpol} and \ref{fact:Lem4GenQCV}.

  Let $\alpha_t = 1/\bm{p}^{\overline{t}}$ and let
  \begin{align}
    c_i = \sum_{0 \le j_1 \le \cdots \le j_{n-i} \le i} \prod_{k=1}^{n-i}
    \left( \frac{1}{\bm{p}^{\overline{j_k}}} - c \bm{q}^{\overline{n-j_k-k}} \right)
  \end{align}
  so that $f(a) = \sum_{i=0}^{n} c_i \prod_{k=0}^{i-1} (a - \alpha_k)$.
  The recurrence \eqref{eq:NewtonInterpolDD} with $c^{(0)}_{i} = c_i$ and $c^{(t)}_{n+1} = 0$ is then solved by
  \begin{align} \label{eq:qCVGenRecurrenceSol}
    c^{(t)}_{i} = \sum_{0 \le j_1 \le \cdots \le j_{n-i} \le i} \prod_{k=1}^{n-i}
    \left( \frac{1}{\bm{p}^{\overline{t+j_k}}} - c \bm{q}^{\overline{n-j_k-k}} \right).
  \end{align}
  Indeed the recurrence \eqref{eq:NewtonInterpolDD} with
  $\alpha_t = 1/\bm{p}^{\overline{t}}$ and \eqref{eq:qCVGenRecurrenceSol} gives
  the identity \eqref{eq:Rec4GenQCV} in Fact \ref{fact:Lem4GenQCV} with parameters
  \begin{align}
    u_j \gets \frac{1}{\bm{p}^{\overline{t+j}}}, \qquad
    v_j \gets c \bm{q}^{\overline{n-j}}, \qquad
    m \gets n-i, \qquad
    n \gets i.
  \end{align}
  Fact \ref{fact:NewtonInterpol} thereby implies that
  \begin{align}
    f(1/\bm{p}^{\overline{t}}) = f(\alpha_t) = c^{(t)}_{0}
    = \prod_{k=1}^{n} \left( \frac{1}{\bm{p}^{\overline{t}}} - c \bm{q}^{\overline{n-k}} \right)
    = \prod_{k=0}^{n-1} \left( \frac{1}{\bm{p}^{\overline{t}}} - c \bm{q}^{\overline{k}} \right)
    = g(1/\bm{p}^{\overline{t}})
  \end{align}
  for $t \in \mathbb{Z}_{\ge 0}$.
  That completes the proof of Lemma \ref{lem:GenQCV}.
\end{proof}

\bibliographystyle{amsplain}
\bibliography{ksh94}

\providecommand{\bysame}{\leavevmode\hbox to3em{\hrulefill}\thinspace}
\providecommand{\MR}{\relax\ifhmode\unskip\space\fi MR }
\providecommand{\MRhref}[2]{%
  \href{http://www.ams.org/mathscinet-getitem?mr=#1}{#2}
}
\providecommand{\href}[2]{#2}
\begin{thebibliography}{10}

\bibitem{Aigner(2007CE)}
M.~Aigner, \emph{A course in enumeration}, Graduate Texts in Mathematics, vol.
  238, Springer, Berlin, 2007.

\bibitem{Aigner-Guenter(2014TheBook)}
M.~Aigner and G.~M. Ziegler, \emph{Proofs from the book}, fifth ed.,
  Springer-Verlag, Berlin, 2014.

\bibitem{Baker-GravesMorris(1996)}
G.~A. Baker~Jr. and P.~Graves-Morris, \emph{{Pad\'e} approximants}, second ed.,
  Encyclopedia of Mathematics and its Applications, vol.~59, Cambridge
  University Press, Cambridge, 1996.

\bibitem{Chihara(1978OP)}
T.~S. Chihara, \emph{An introduction to orthogonal polynomials}, Mathematics
  and its Applications, vol.~13, Gordon and Breach Science Publishers, New
  York--London--Paris, 1978.

\bibitem{Foata(1984)}
D.~Foata, \emph{Combinatoire des identit\'es sur les polyn\^omes orthogonaux},
  Proceedings of the International Congress of Mathematicians, Vol. 1, 2
  (Warsaw, 1983), PWN, Warsaw, 1984, pp.~1541--1553.

\bibitem{Gansner(1981Burge)}
E.~R. Gansner, \emph{The enumeration of plane partitions via the {Burge}
  correspondence}, Illinois J. Math. \textbf{25} (1981), 533--554.

\bibitem{Gansner(1981HG)}
\bysame, \emph{The {Hillman}-{Grassl} correspondence and the enumeration of
  reverse plane partitions}, J. Combin. Theory Ser. A \textbf{30} (1981),
  71--89.

\bibitem{Gessel-Viennot(1985)}
I.~Gessel and G.~Viennot, \emph{Binomial determinants, paths, and hook length
  formulae}, Adv. in Math. \textbf{58} (1985), 300--321.

\bibitem{Gessel-Viennot(PRE1989)}
I.~M. Gessel and X.~G. Viennot, \emph{Determinants, paths and plane
  partitions},  (preprint).

\bibitem{Ismail(2005CQOP)}
M.~E.~H. Ismail, \emph{Classical and quantum orthogonal polynomials in one
  variable}, Encyclopedia of Mathematics and its Applications, vol.~98,
  Cambridge University Press, Cambridge, 2005.

\bibitem{Johansson(2002)}
K.~Johansson, \emph{Non-intersecting paths, random tilings and random
  matrices}, Probab. Theory Related Fields \textbf{123} (2002), 225--280.

\bibitem{Kamioka(2007)}
S.~Kamioka, \emph{A combinatorial representation with {Schr\"oder} paths of
  biorthogonality of {Laurent} biorthogonal polynomials}, Electron. J. Combin.
  \textbf{14} (2007), Research Paper 37, 22 pp. (electronic).

\bibitem{Kamioka(2008)}
\bysame, \emph{A combinatorial derivation with {Schr\"oder} paths of a
  determinant representation of {Laurent} biorthogonal polynomials}, Electron.
  J. Combin. \textbf{15} (2008), Research Paper 76, 20 pp. (electronic).

\bibitem{KimD(1992)}
D.~Kim, \emph{A combinatorial approach to biorthogonal polynomials}, SIAM J.
  Discrete Math. \textbf{5} (1992), 413--421.

\bibitem{Koekoek-Leskey-Swarttow(2010)}
R.~Koekoek, P.~A. Lesky, and R.~F. Swarttouw, \emph{Hypergeometric orthogonal
  polynomials and their $q$-analogues}, Springer Monographs in Mathematics,
  Springer-Verlag, Berlin, 2010.

\bibitem{Krattenthaler(1990)}
C.~Krattenthaler, \emph{Generating functions for plane partitions of a given
  shape}, Manuscripta Math. \textbf{69} (1990), 173--201.

\bibitem{Lindstroem(1973)}
B.~Lindstr{\"o}m, \emph{On the vector representations of induced matroids},
  Bull. London Math. Soc. \textbf{5} (1973), 85--90.

\bibitem{MacMahon(1916)}
P.~A. MacMahon, \emph{Combinatory analysis}, vol.~2, Cambridge University
  Press, Cambridge, 1916.

\bibitem{Maeda-Miki-Tsujimoto(2013)}
K.~Maeda, H.~Miki, and S.~Tsujimoto, \emph{From orthogonal polynomials to
  integrable systems}, Trans. Jpn. Soc. Ind. Appl. Math. \textbf{23} (2013),
  341--380 (Japanese).

\bibitem{Stanley(1971)}
R.~P. Stanley, \emph{Theory and application of plane partitions, {I}, {II}},
  Studies in Appl. Math. \textbf{50} (1971), 167--188, 259--279.

\bibitem{Stanley(1973)}
\bysame, \emph{The conjugate trace and trace of a plane partition}, J.
  Combinatorial Theory Ser. A \textbf{14} (1973), 53--65.

\bibitem{Szego(1975OP)}
G.~Szeg{\H o}, \emph{Orthogonal polynomials}, fourth ed., American Mathematical
  Society, Colloquium Publications, vol.~23, American Mathematical Society,
  Providence, R.I., 1975.

\bibitem{Viennot(1983OP)}
G.~Viennot, \emph{Une th\'eorie combinatoire des polyn\^omes orthogonaux
  g\'en\'eraux}, Universit\'e du Qu\'ebec \`a Montr\'eal, 1983.

\bibitem{Viennot(1985)}
\bysame, \emph{A combinatorial theory for general orthogonal polynomials with
  extensions and applications}, Orthogonal Polynomials and Applications
  (Bar-le-Duc, 1984), Lecture Notes in Math., vol. 1171, Springer, Berlin,
  1985, pp.~139--157.

\end{thebibliography}

\end{document}